\documentclass[12pt,reqno]{amsart}

\usepackage{amsfonts,amsmath,amsthm}
\usepackage{amssymb,epsfig}
\usepackage{cases}
\usepackage[usenames]{color}
\usepackage{enumerate} 

\usepackage{color} 
\definecolor{vert}{rgb}{0,0.6,0}

\theoremstyle{plain}
\newtheorem{thm}{Theorem}

\newtheorem{defn}{Definition}

\newtheorem{ex}{Example}
\newtheorem{lem}[thm]{Lemma}
\newtheorem{cor}[thm]{Corollary}
\newtheorem{prop}[thm]{Proposition}
\theoremstyle{remark}
\newtheorem{rem}{\bf{Remark}}

\newcommand{\N}{\mathbb{N}}
\newcommand{\R}{\mathbb{R}}
\newcommand{\bS}{\mathbb{S}}

\newcommand{\cA}{\mathcal{A}}
\newcommand{\cB}{\mathcal{B}}
\newcommand{\cH}{\mathcal{H}}
\newcommand{\cL}{\mathcal{L}}

\newcommand{\cS}{\mathcal{S}}

\newcommand{\Li}{L^{\infty}}
\newcommand{\W}{W^{1,\infty}}

\newcommand{\Q}{\mathbb{R}^N\times(0,T)}
\newcommand{\cQ}{\mathbb{R}^N\times[0,T]}

\newcommand{\al}{\alpha}
\newcommand{\gam}{\gamma}
\newcommand{\del}{\delta}
\newcommand{\ep}{\varepsilon}
\newcommand{\kap}{\kappa}
\newcommand{\lam}{\lambda}
\newcommand{\sig}{\sigma}
\newcommand{\om}{\omega}

\newcommand{\Del}{\Delta}
\newcommand{\Gam}{\Gamma}
\newcommand{\Lam}{\Lambda}
\newcommand{\Om}{\Omega}

\newcommand{\ol}{\overline}
\newcommand{\ul}{\underline}
\newcommand{\pl}{\partial}
\newcommand{\supp}{{\rm supp}\,}

\newcommand{\Div}{{\rm div}\,}
\newcommand{\tr}{{\rm tr}\,}
\newcommand{\Per}{{\rm Per}\,}

\begin{document}
\title[Short Time Uniqueness]{
Short Time Uniqueness Results for Solutions of \\
Nonlocal and Non-monotone Geometric Equations
}

\author[G. BARLES, O. LEY and H. MITAKE]
{Guy BARLES, Olivier LEY and Hiroyoshi MITAKE}

\address[G. Barles]
{Laboratoire de Math\'ematiques et Physique 
Th\'eorique, F\'ed\'eration Denis Poisson, 
Universit\'e de Tours, 
Place de Grandmont, 
37200 Tours, FRANCE}
\email{barles@lmpt.univ-tours.fr}
\urladdr{http://www.lmpt.univ-tours.fr/~barles}

\address[O. Ley]
{IRMAR, INSA de Rennes, F-35043 Rennes, France} 
\email{olivier.ley@insa-rennes.fr}
\urladdr{http://www.lmpt.univ-tours.fr/~ley}

\address[H. Mitake]
{Department of Applied Mathematics,
Graduate School of Engineering
Hiroshima University
Higashi-Hiroshima 739-8527, Japan}
\email{mitake@amath.hiroshima-u.ac.jp}

\keywords{Nonlocal Hamilton-Jacobi Equations, Nonlocal Front Propagation, 
Short Time Uniqueness, Non-Fattening Condition, Lower Gradient Estimate, 
Dislocation Dynamics, Fitzhugh-Nagumo System, Viscosity Solution}
\subjclass[2010]{
35K15, 
34A12, 
35A02, 
49L25 
45K05, 
53C44 
}

\thanks{This work was partially supported by the ANR (Agence Nationale de la Recherche) through MICA project (ANR-06-BLAN-0082) and 
by the Research Fellowship (20-5332, 22-1725)
for Young Researcher from JSPS and 
Excellent Young Researchers Overseas Visit Program of JSPS}

\date{\today}

\begin{abstract}
We describe a method to show 
short time uniqueness results 
for viscosity solutions of general nonlocal 
and non-monotone second-order geometric equations 
arising in front propagation problems. 
Our method is based on some lower gradient bounds
for the solution. These estimates are crucial to
obtain regularity properties of the front,
which allow to deal with nonlocal terms in the equations.
Applications to 
short time  uniqueness results 
for the initial value problems for 
dislocation type equations, 
asymptotic equations of a FitzHugh-Nagumo type system 
and equations depending on the Lebesgue measure of the fronts
are presented.
\end{abstract}

\maketitle


\section{Introduction}

We are concerned with 
the evolution of compact hypersurfaces 
$\{\Gam_{t}\}_{t\ge0}\subset\R^{N}$ 
moving according to the general non-local law of propagation  
\begin{equation}\label{general law}
V=h(x,t,\Om_{t},n(x),Dn(x)) 
\quad\textrm{on} \ \Gam_{t}, 
\end{equation}
where $V$ is the normal velocity of $\Gam_{t}$ which depends, through the evolution law
$h$, on time, on the position of $x \in \Gam_{t}$, on the set $\Om_{t}$ enclosed by $\Gam_{t}$,
on the unit normal $n(x)$ to $\Gam_{t}$ at $x$
pointing outward to $\Om_{t}$ and on its gradient $Dn(x)$ which carries the curvature dependence of the velocity.

When such motion is local, i.e., when $h$ does not depend on $\Om_{t}$, and satisfies \textit{the inclusion principle} or \textit{geometrical monotonicity}, 
i.e., when, at least formally, the inclusion $\Om_{0}^{1}\subset \Om_{0}^{2}$ at time $t=0$ implies
$\Om_{t}^{1}\subset \Om_{t}^{2}$ for any $t>0$, it is proved by Souganidis and the first author \cite{BS} that
the motion can be defined and studied by the  \textit{level set approach}, which was introduced by
Osher and Sethian \cite{OS} for numerical calculations and then
developed, from a theoretical point of view, by Evans and Spruck \cite{ES} for the mean curvature
motion and by Chen, Giga and Goto \cite{CGG} for general velocities. 
This approach replaces the geometrical problem~\eqref{general law}
with a degenerate parabolic partial differential equation called the \textit{geometric}
or \textit{level set equation.} This equation is designed
to describe the desired evolution via the 0-level set of its
solution. More precisely, the existence and
uniqueness of the level set solution $u : \R^N \times [0,+\infty)\to\R$
allows to define  $\Gam_{t}$
as being the set $\{x\in\R^N \mid u(\cdot,t) = 0\}$.

In recent years, there has been much interest on the study of 
front propagations problems in cases when the normal velocity 
of the front depends on a non-local way of the enclosed 
region like \eqref{general law}.
This interest was motivated by several types of applications 
like dislocations' theory or 
FitzHugh-Nagumo type systems or volume dependent velocities 
that we describe below. It is worth pointing out that the level set approach still applies for motions with nonlocal velocities provided that the inclusion principle holds, following the ideas of Slepcev \cite{Sl}. But, in many of the above mentioned applications, one faces non-monotone surface evolution equations. For such class of problems, the level set approach cannot be used directly since the classical comparison arguments of viscosity solutions' theory fail and therefore, the existence and uniqueness of viscosity solutions to these equations become an issue.

Though the existence properties for such motions 
seem now to be well understood (see \cite{GGI, SS, BCLM3}), 
this is not the case for uniqueness.  
In particular, there are not many uniqueness results for 
curvature dependent velocities. 
As far as the authors know, there are only two works 
by Forcadel \cite{F} and Forcadel and Monteillet \cite{FM} 
which investigate the motion arising  in a model for dislocation dynamics 
which is included by our general equations. 
The aim of this article is to consider cases where we have, at the same time, a non-local velocity which induces a non-monotone evolution together with a curvature dependence (we explain later on the state of the art for such problems and why the curvature dependence creates a specific difficulty). More specifically, we describe a method to show short time uniqueness results 
for the general motion \eqref{general law}.

We now describe some typical applications we have in mind. 
We first consider a model 
for dislocation dynamics 
\begin{equation}\label{dislocation model}
h=M(n(x))
\bigl(
c_{0}(\cdot,t)\ast\mathbf{1}_{\ol{\Om}_{t}}(x)+c_{1}(x,t)
-\Div_{\Gam_{t}}\xi(n(x))
\bigr), 
\end{equation}
where $\mathbf{1}_A$ denotes the indicator function of 
a subset $A$ of $\R^N$ and $M, \xi:\bS^{N-1}\to\R$, 
$c_{0}, c_{1}:\cQ\to\R$ are given functions 
and 
we write 
\[
c_{0}(\cdot,t)\ast\mathbf{1}_{\ol{\Om}_{t}}(x)
:=\int_{\R^N}c_{0}(x-y,t)\mathbf{1}_{\ol{\Om}_{t}}(y)\,dy.  
\]
Here, $\bS^{N-1}$ denotes the $(N-1)$-dimensional unit sphere.
The term $\Div_{\Gam_{t}}\xi(n(x))
:=\tr\bigl((I-n(x)\otimes n(x))D_{x}\xi(n(x))\bigr)$ 
is called the \textit{anisotropic} (or \textit{weighted}) 
\textit{mean curvature} of $\Gam_{t}$ at $x$ 
(in the direction of $n(x)$). 
See for instance Giga \cite{G}. 
Typically, the reasonable assumptions in this context 
are the following: 
$M$ is a positive and bounded function, 
$c_{0}, c_{1}$ are bounded, continuous functions 
which are Lipschitz continuous in $x$ variable 
(uniformly with respect to $t$ variable), 
$c_{0},D_{x}c_{0}\in\Li([0,T];L^{1}(\R^N))$ 
and $\xi$ is a positively homogeneous function with degree $0$. 
The surface evolution equation \eqref{dislocation model} 
without the last term in the right hand side 
is well-known as typical models of the dislocation dynamics 
(see \cite{RBF, AHBM} for a derivation and 
the physical background).

We next consider 
asymptotic equations of a FitzHugh-Nagumo type system 
as an example of 
interface dynamics coupled with a diffusion equations, 
\begin{gather}
h=
\al(v(x,t))-\Div_{\Gam_{t}}(n(x))\nonumber\\
\textrm{and}\label{FN model}\\
v_{t}-\Del v
=g^{+}(v)\mathbf{1}_{\ol{\Om}_{t}}+
g^{-}(v)(1-\mathbf{1}_{\ol{\Om}_{t}}), \nonumber 
\end{gather}
where 
$\al, g^{\pm}:\R\to\R$ are bounded and Lipschitz continuous 
with $g^{-}\le g^{+}$.  
This system has been investigated by 
Giga, Goto and Ishii \cite{GGI} and 
Soravia and Souganidis \cite{SS}.

Finally, 
we consider equations depending on the measure of the fronts like 
\begin{equation}\label{measure model}
h=
\beta(\cL^{N}(\ol{\Om}_{t}))
-\Div_{\Gam_{t}}(n(x)), 
\end{equation}
where the function $\beta:\R\to\R$ is 
Lipschitz continuous. 
A typical example is 
$\beta(r)=a+br$ for some $a,b\in\R$ 
which has been investigated by Chen, Hilhorst and Logak in 
\cite{CHL} (see also \cite{C, CP}).

As we already mentioned it above, 
these examples are not only nonlocal but also 
non-monotone surface evolution equations. 
Indeed, in \eqref{dislocation model}, 
the kernel $c_{0}$ may change sign and, 
in \eqref{FN model} and in \eqref{measure model},  
the functions $\al$, $\beta$ may be non-monotone.
We also refer to \cite{C, Sl, DKS, Sr} for some 
monotone non-local geometric equations.  
By using the framework which we present in this paper, 
we give short time uniqueness results 
for \eqref{dislocation model}, \eqref{FN model} and 
\eqref{measure model}.


There are many results of 
existence and uniqueness for 
the simplest case of motions of \eqref{dislocation model}, 
i.e.,  $M(p)\equiv 1$ and 
$\xi(p)=p/|p|$ without a curvature term. 
A short time existence and uniqueness result was first 
obtained in \cite{AHBM}. 
But then most of the results were obtained for curvature-independent
velocities ($\xi =0$):
long time existence and uniqueness results
were obtained when the velocity is positive, i.e.,
\begin{equation}\label{vitpos}
h>0 
\quad\textrm{on} \ \Gam_{t}, 
\end{equation}
by Alvarez, Cardaliaguet and Monneau 
in \cite{ACM} 
and by the first two authors in \cite{BL} 
by different methods. 
The first two authors with Cardaliaguet and Monneau 
in \cite{BCLM1} presented 
a new notion of weak solutions 
(see Definition \ref{def weak sol}) of 
the level set equation for \eqref{dislocation model} 
without a curvature term, 
gave the global existence of these weak solutions 
and analysed the uniqueness of them when~\eqref{vitpos} holds. 
A similar concept of solutions already appeared in 
\cite{GGI, SS}. 
In the companion paper \cite{BCLM2}, 
the first two authors with Cardaliaguet and Monteillet 
proposed a new perimeter estimate for the evolving 
fronts with uniform interior cone property and 
by using this, 
they extended the uniqueness result for 
dislocation dynamics equations and 
provided the uniqueness result for 
asymptotic equations of a FitzHugh-Nagumo type system,
still under the positiveness assumption~\eqref{vitpos}.  

In this paper, 
we do not use the perimeter estimate in an essential way 
but either elementary measure estimates or, 
in the most sophisticated cases, the interior cone property 
(see Lemma \ref{lem-for-unique}). 
Since the studies by 
\cite{ACM, BL, BCLM1, BCLM2}, 
it is now well-known that 
estimates on \textit{lower gradient bound} 
and perimeter of $0$-level sets of viscosity solutions 
of associated local equations are 
key properties to obtain existence and uniqueness results 
for nonlocal equations derived from \eqref{dislocation model},  
\eqref{FN model} and \eqref{measure model}. 
Let us describe the main difficulty of 
our problem and, 
to do so, we consider the level set equations of 
the simplest case of \eqref{dislocation model} 
or \eqref{FN model} here. 
Considering the non-local part as a given function, 
we are led to the study 
the (local) initial value problem 
\begin{equation}\label{local}
\left\{
\begin{aligned}
&
u_t=
\Bigl(c(x,t)
+\Div\bigl(\frac{Du}{|Du|}\bigr)\Bigr)|Du|
&& \textrm{in} \ \Q, \\
&u(\cdot,0)=u_0
&& \textrm{in} \ \R^{N}, 
\end{aligned}
\right.
\end{equation}
where 
$u_{0}\in\W(\R^N)$ and $c\in C(\cQ)$ are bounded 
and Lipschitz continuous 
with respect to the $x$ variable. 
One of our main results is 
a short time lower gradient bound estimate 
for the viscosity solution of \eqref{local}, i.e.,  
\begin{equation}\label{lower}
|Du(x,t)| \ge \eta(t) > 0 
\quad\textrm{in a neighborhood of} \ \{u(\cdot,t)=0\}. 
\end{equation}
For first-order eikonal equations, 
lower gradient bound comes naturally from the Barron-Jensen's 
approach (see \cite{L}). 
For second-order equations like \eqref{local}, 
it is affected by the ``\textit{diffusion}" term and 
the \textit{non-empty interior} difficulty 
and therefore we cannot expect that  
the property \eqref{lower} holds generally and for long-time. 
Indeed, in~\cite{BP}, see also~\cite{koo99, gk01}, they consider the simple
example of~\eqref{local} with $c\equiv 1$ and 
smooth $u_0$ such that $Du_0\not= 0$ on the initial
front $\{u_0=0\}.$ They prove that, up to choose suitable $u_0$, 
fattening may occur for arbitrary $t>0,$ i.e., the front may develop
an interior.
It is precisely this reason which implies that
there are not many results on 
the nonlocal second-order equations  
like the level sets equations of 
\eqref{dislocation model},  
\eqref{FN model} and \eqref{measure model}
and a short-time result is optimal.

Existence results were obtained 
by \cite{GGI, SS, BCLM3} 
but they concern merely existence of weak solutions 
defined by Definition \ref{def weak sol}. 
As stated above, there are two works \cite{F,FM} which 
give uniqueness results for the motion \eqref{dislocation model}. 
The difference between our results and theirs is that, 
in \cite{F}, only the evolution of hypersurfaces which can be expressed 
by graphs of functions is considered while, in \cite{FM}, the arguments 
are based on minimizing movement for \eqref{dislocation model} and 
they are completely different from our arguments 
which are based on the theory of viscosity solutions. 
Moreover, for the existence of minimizing movement for 
the simplest case of \eqref{dislocation model}, 
the assumptions that $c_{0}(\cdot,t)$ is symmetry and 
$c_{0}, c_{1}$ are smooth enough 
are essentially used. 
Therefore, uniqueness results for the examples \eqref{FN model} 
and \eqref{measure model} are not 
covered by \cite{FM}. 

Another difference with existing results in the literature
is that $h$ is allowed to change sign 
in~\eqref{general law}, contrary to~\cite{ACM, BL, BCLM1, BCLM2}
where~\eqref{vitpos} is one of the main assumption to get uniqueness.
It may give rise of fattening, see~\cite[Proposition 4.4]{BSS},
and it is another explanation of the short time result.

Finally, we explain the key idea to obtain~\eqref{lower} 
for viscosity solutions of~\eqref{local}. 
In order to get it, we make the following assumption 
on $u_{0}$. 
There exist constants $\lam_0\in(0,1)$, 
$\eta_0>0$ and $\nu\in C(\R^{N},\R^{N})$ such that 
\[
u_0(x+\lam\nu(x))\ge u_0(x)+\lam\eta_{0}
 \ 
\textrm{in a neighborhood of} \  \{u_{0}(\cdot)=0\} \ 
\]
for all 
$\lam\in[0,\lam_0]$. 
Then we prove that such a property 
is preserved for the solution of 
\eqref{local}, at least 
for short time, i.e., 
\begin{equation}\label{preserving}
u(x+\lam\nu(x), t)\ge 
u(x,t)+\lam\eta(t) 
\ \textrm{in a neighborhood of} \  \{u(\cdot,t)=0\}  
\end{equation}
for all $\lam\in[0, \ol{\lam}]$, 
$t\in[0,\ol{t}\land T]$ and 
some $\ol{t}>0$, $\ol{\lam}\in(0, \lam_{0}]$, 
where $\eta:[0,\ol{t}\land T]\to[0,\infty)$ 
is a non-increasing continuous function such that 
\begin{equation}\label{lower lem}
\eta(t)>0 \ \textrm{for all} \ t\in[0,\ol{t}\land T). 
\end{equation}
The assumption on $u_0$ is 
inspired by \cite[Theorem 4.3]{BSS}, 
where it is formulated only for the sign-distance 
function. A similar result to~\eqref{preserving} may be found
in~\cite{bcl08} where it is used to prove uniqueness results for the
mean curvature motion for entire graphs.

If $u_0$ is a smooth function with $Du_0 \neq 0$ on the compact hypersurface $\Gamma_0=\{x:\, u_0(x)=0\}$, then the assumption is satisfied with $\nu(x)=Du_0 (x)$ and if there exists a smooth solution of the level set equation, then \eqref{preserving} holds for short time. But, on one hand, the general degenerate parabolic and nonlinear equations we consider do not have classical solutions in general, and on the other hand, the above assumption on $u_0$ is valid in cases when $\Gamma_0$ is not a smooth hypersurface, which is also an important point here.

The proof of~\eqref{preserving} uses in a crucial way 
the geometric property of \eqref{local} 
and a \textit{continuous dependence result} 
for parabolic problems (which is, 
by the way, of independent interest). 
We refer to \cite{JK1, JK2, BJ} and references therein 
for the detail of the continuous dependence result 
for elliptic and parabolic problems.

We derive lower gradient estimate \eqref{lower} from 
\eqref{preserving} formally here. 
We have 
\begin{align*}
\lam\eta(t_{0})
\le \ &
u(x_{0}+\lam \nu(x_0),t_{0})-u(x_{0},t_{0})\\
{}= \ &
\lam \langle Du(x_{0},t_{0}),\nu(x)\rangle
+o(\lam\|\nu\|_{\infty})\\
{}\le \ &
\lam|Du(x_{0},t_{0})|\|\nu\|_{\infty}+o(\lam\|\nu\|_{\infty}) 
\quad\textrm{in a neighborhood of} \ 
\{u(\cdot,t)=0\} 
\end{align*}
for all $t\in[0,\ol{t}\land T]$ 
with $o(r)/r\to0$ as $r\to0$. 
Dividing $\lam$ in the above and 
taking a sufficiently small $\lam\in(0,\ol{\lam}]$,  
we get the lower estimate \eqref{lower}. 
We also obtain the interior cone property of 
fronts by \eqref{preserving}.

\medskip

The paper is organized as follows: 
in Section 2, 
we state a continuous dependence result for 
a class of equations which encompasses
level set equations associated 
to~\eqref{general law}. 
In Section 3, we obtain the key estimate~\eqref{preserving}
and derive the lower-bound gradient and perimeter estimates of 
$0$-level sets of viscosity solutions of local equations. 
In Section 4, 
we consider the \textit{level set} equation of 
\eqref{general law} and give the proof for the short time 
uniqueness result (Theorem \ref{uniqueness}). 
Section 5 is devoted to
existence and uniqueness results 
for the level set equations of 
\eqref{dislocation model}, 
\eqref{FN model} and \eqref{measure model} 
as applications of Theorem \ref{uniqueness}.

\medskip
\noindent
{\bf Notations.} 
For some $k\in\N$, we denote by $\R^{k}$ 
the $k$-dimensional Euclidean space
equipped with the usual Euclidean inner product
 $\langle\cdot,\cdot\rangle$,
and by $\cS^{k}$ the space of $k\times k$ symmetric matrices. 
We write $B(x,r)=\{y\in\R^{k}\mid |x-y|<r\}$ for $x\in\R^{k},$ 
$r>0,$ and 
$A+rB(0,1):=\{x+y\mid x\in A, y\in B(0,r)\}$ 
for $A\subset\R^{k}$.   
The symbols $\cL^{k}(A)$ and $\cH^{k}(A)$ denote the $k$-dimensional Lebesgue 
and Hausdorff measures, respectively. 
We write $X^{T}$ for the transpose of the matrix $X$ and 
$|X|=\sup\{|X\xi|\mid\xi\in\R^{k}, |\xi|=1\}$. 
Finally, for $a,b\in\R$, 
we write $a\land b=\min\{a,b\}$ and 
$a\vee b=\max\{a,b\}$. 

\medskip
\noindent
\textbf{Acknowledgements. }
We are grateful 
to A.~Chambolle, E.~Jakobsen and L.~Rifford 
for their comments and advice. 
This work was done while H. Mitake was visiting 
the Laboratoire de 
Math\'ematiques et Physique Th\'eorique, Universit\'e 
de Tours. 
His grateful thanks go to the faculty and staffs.


\section{Continuous Dependence of Solutions}
In this section, we are concerned with the equation 
\begin{equation}\label{HJ}
u_t+H(x,t,Du,D^{2}u)=0
\quad\textrm{in} \ \Q, 
\end{equation}
$T>0$, $u:\Q\to\R$ is the unknown function, 
$u_t$, $Du$ and $D^{2}u$ stand respectively for its time 
and space derivatives, 
and Hessian matrix with respect to $x$ variable. 
We use the following assumptions. 
\begin{itemize}
\item[{\rm (A1)}]
$H\in C(\R^N\times[0,T]\times(\R^{N}\setminus\{0\})
\times\cS^{N})$. 

\item[{\rm (A2)}]
The equation is \textit{degenerate parabolic}, i.e., 
\[H(x,t,p,X)\ge H(x,t,p,Y), \] 
for any $(x,t,p)\in\R^N\times[0,T]\times(\R^{N}\setminus\{0\})$ 
and $X,Y\in\cS^{N}$ with $X\le Y$, 
where $\le$ stands for the usual partial ordering for 
symmetric matrices.

\item[{\rm (A3)}]
For any 
$(x,t)\in\cQ$, 
$H^{\ast}(x,t,0,0)=H_{\ast}(x,t,0,0)$, 
where 
$H^{\ast}$ (resp., $H_{\ast}$) is 
the upper-semicontinuous envelope 
(resp., lower semicontinuous envelope) 
of $H$. 

\item[{\rm (A4)}] There exist 
$\kap_{1},\kap_{2}\ge0$, 
$M\ge0$ such that  
\begin{align}
{}&
H_{2}(y,t,p,Y)
-
H_{1}(x,t,p,X)\nonumber\\
{}\le\, &
C_{R}\bigl(\frac{|x-y|^{4}}{\ep^{4}}+\kap_{1}
+\frac{\kap_{2}|x-y|^{2}}{\ep^{4}}+
\rho\|A^{2}\|
\bigr)
\label{assumption conti}
\end{align}
for any $\rho, \ep\in(0,1)$, $R>0$, 
$x,y\in \ol{B}(0,R)$, 
$p=4\ep^{-4}|x-y|^{2}(x-y)$,  
$X,Y\in\cS^{N}$ and some $C_{R}>0$ satisfying 
\begin{equation}\label{matrix ineq}
\left\{
\begin{aligned}
&
|p|\le M, \\
&
\left(
\begin{array}{cc}
X & 0 \\
0 & -Y 
\end{array}
\right)
\le 
A+\rho A^{2},  
\end{aligned}
\right.
\end{equation}
where 
\begin{equation}\label{matrix-A}
A:=
\frac{|x-y|^2}{\ep^{4}} 
\left(
\begin{array}{cc}
I & -I \\
-I & I 
\end{array}
\right)
+\frac{|x-y|^2 }{\ep^{4}}
\left(
\begin{array}{cc}
\hat p \otimes\hat p & 
-\hat p\otimes\hat p \\
-\hat p\otimes\hat p & 
\hat p\otimes\hat p
\end{array}
\right),
\end{equation}
with $\hat p :=p/|p|$.

\end{itemize}
We note that, in this section, we do not assume that 
$H$ is geometric. 

\begin{thm}\label{continuous dependence}
Let $H_{1}, H_{2}$ be functions on 
$\R^N\times[0,T]\times(\R^{N}\setminus\{0\})
\times\cS^{N}$ satisfying assumptions 
{\rm (A1)--(A4)}. Let $u_1,u_2\in C(\cQ)$ be, respectively, a bounded viscosity subsolution 
and viscosity supersolution of \eqref{HJ} with $H=H_{i}$ 
for $i=1,2$. 
Assume that there exists $L=L_{u_i}>0$ such that
\begin{equation}\label{hyplip}
|u_{i}(x,t)-u_{i}(y,t)|\le L|x-y| 
\quad {\rm for \ all \ } x,y\in\R^N, \, t\in[0,T] 
\end{equation}
for either $i=1$ or $2$, and that there exists $R>0$ such that
\begin{equation}\label{umoins1}
u_i(x,t)=-1 \quad {\rm for \ all \ } x\in\R^N\setminus B(0,R), \, t\in[0,T] 
\end{equation}
for both $i=1$ and $2.$
Then there exists $M_{1}>0$ which depends only on 
$C$, $L$, 
$\|u_1\|_{\infty}$ and $\|u_2\|_{\infty}$ 
such that 
\begin{equation}\label{continuous dependence ineq}
\sup_{x\in\R^{N}}(u_1-u_2)(x,t)\le 
\sup_{x\in\R^{N}}(u_1-u_2)(x,0)\\
+M_{1}\bigl(\kap_{1}t+
(\kap_{2}t)^{1/2}
\bigr)
\end{equation}
for all $t\in[0,T]$. 
\end{thm}

\begin{rem}
An assumption like (A4) is natural in viscosity
theory to obtain continuous dependence results of the type
\eqref{continuous dependence ineq} and the regularity of the solution (cf. \eqref{hyplip}) is a key ingredient too, see \cite{BJ, JK1, JK2}.
In Example \ref{ex-1} below, we show that (A4)
holds in the cases we are interested in. 
Note that \eqref{umoins1} are not restrictive
assumptions when dealing with front propagation problems, see
\cite{G, BL, BCLM2, BCLM3}. 
\end{rem}

\begin{proof}
Let $\ep\in(0,1)$ and $K>0$. 
We shall later fix $\ep, K$. 
Consider 
\[
\sup_{x,y\in\R^N,t\in[0,T]}
\{u_1(x,t)-u_2(y,t)-
\frac{|x-y|^4}{\ep^4}-Kt\}. 
\]
Noting \eqref{umoins1}, 
it is clear that 
the supremum is attained at 
$(\ol{x},\ol{y},\ol{t})\in 
\ol{B}(0,R+1)^{2}\times[0,T]$  
for small $\ep>0$. 

We consider the case where $\ol{t}\in(0,T]$. 
In view of Ishii's Lemma, for any $\rho>0$, 
there exist 
$(a,p,X)\in \ol{J}^{2,+}u_{1}(\ol{x},\ol{t})$ and 
$(b,p,Y)\in \ol{J}^{2,-}u_{2}(\ol{y},\ol{t})$ 
(see \cite{CIL} for the notation)
such that 
\begin{gather}\label{ishii-lem}
a-b\ge K, \quad 
p=\frac{4|\ol{x}-\ol{y}|^2}{\ep^{4}}(\ol{x}-\ol{y}), \ 
\left(
\begin{array}{cc}
X & 0 \\
0 & -Y 
\end{array}
\right)
\le 
A+\rho A^{2}, 
\end{gather}
where $A$ is the matrix defined by \eqref{matrix-A}. 
The definition of viscosity solutions immediately implies the
following inequalities: 
\[
a+(H_{1})_{\ast}(\ol{x},\ol{t},p,X)\le 0, 
\quad
b+(H_{2})^{\ast}(\ol{y},\ol{t},p,Y)\ge 0.
\]
Hence we have 
\begin{equation}\label{ineq-1}
K+(H_{1})_{\ast}(\ol{x},\ol{t},p,X)-(H_{2})^{\ast}(\ol{y},\ol{t},p,Y)\le0. 
\end{equation}
Using that $u_{1}$ or $u_{2}$ is Lipschitz continuous with respect to 
$x$ variable, 
we get, by standard estimates, 
\[
|\ol{x}-\ol{y}|\le M\ep^{4/3}, \ \ 
|p|\le M, \ 
\]
where $M$ is a positive constant which depends 
only on 
$\|u_1\|_{\Li(\cQ)}$, 
$\|u_2\|_{\Li(\cQ)}$ and 
$L$.

We now distinguish two cases: 
(i) 
for any $\ep\in(0,1)$, 
$p\not=0$; 
(ii) 
there exist 
$\{\ep_j\}_{j\in\N}$ 
such that 
$\ep_j\to0$ as $j\to\infty$, 
$p=0$ 
for any $j\in\N$.

We first consider case (i). 
In view of (A4), we have 
\begin{align*}
K
\le \ & 
H_{2}(\ol{y},\ol{t},p,Y)
-
H_{1}(\ol{x},\ol{t},p,X)\\
{}\le \, &
C_{R}
\bigl(\frac{|\ol{x}-\ol{y}|^{4}}{\ep^{4}}+\kap_{1}
+\frac{\kap_{2}|\ol{x}-\ol{y}|^{2}}{\ep^{4}}+
\rho\|A^{2}\|
\bigr). 
\end{align*}
Sending $\rho\to0$, 
we get 
\begin{align*}
K\le \, & 
C_{R}
\bigl(\frac{|\ol{x}-\ol{y}|^{4}}{\ep^{4}}+\kap_{1}
+\frac{\kap_{2}|\ol{x}-\ol{y}|^{2}}{\ep^{4}}\bigr)\\
{}\le  \, & 
\tilde{C}\bigl(\ep^{4/3}+\kap_{1}
+\frac{\kap_{2}}{\ep^{4/3}}\bigr)
=:C_{\ep} 
\end{align*}

In case (ii), we have $\ol{x}=\ol{y}$. 
Due to \eqref{ishii-lem},  
we have $A=0$, $X\le0$ and $Y\ge0$. 
By (A2), we have 
\[
(H_{1})_{\ast}(\ol{x},\ol{t},0,X)\ge 
(H_{1})_{\ast}(\ol{x},\ol{t},0,0)
\ \textrm{and} \ 
(H_{2})^{\ast}(\ol{y},\ol{t},0,Y)\le 
(H_{2})^{\ast}(\ol{y},\ol{t},0,0). 
\]
Therefore, we get 
\begin{align*}
K\le \, & 
(H_{2})^{\ast}(\ol{y},\ol{t},0,Y)
-
(H_{1})_{\ast}(\ol{x},\ol{t},0,X)\\
{}\le \, & 
(H_{2})^{\ast}(\ol{y},\ol{t},0,0)
-
(H_{1})_{\ast}(\ol{x},\ol{t},0,0)
=0. 
\end{align*}

Set $K=C_{\ep}+\tilde{C}\ep^{4/3}$ and then 
the two above cases cannot hold; this means that 
necessarily we have $\ol{t}=0$. 
Therefore, for any $(x,t)\in \ol{B}(0,R)\times[0,T]$, 
\begin{align*}
{}&
(u_1-u_2)(x,t)\\
\le &
u_1(\ol{x},0)-u_2(\ol{y},0)
+Kt\\
\le &
\sup_{x\in\R^{N}}(u_1-u_2)(x,0)+
ML\ep^{4/3}
+\tilde{C}(2\ep^{4/3}+\kap_{1}+\frac{\kap_{2}}{\ep^{4/3}})t. 
\end{align*}
An optimization with respect to $\ep>0$ yields 
\[
(u_1-u_2)(x,t)\le 
\sup_{x\in\R^{N}}(u_1-u_2)(\cdot,0)
+M_{1}\bigl(\kap_{1}t+
(\kap_{2}t)^{1/2}
\bigr)\]
for some 
$M_{1}=
M_{1}(C_{R},L, 
\|u_{1}\|_{\infty},\|u_{2}\|_{\infty})>0$. 
\end{proof}

\begin{ex}\label{ex-1}
{\rm 
We consider the functions $H_{i}:\cQ\times(\R^N\setminus\{0\})
\times\cS^{N}$ defined by 
\begin{equation}\label{Hi-ex1}
H_{i}(x,t,p,X)=
\inf_{\al\in\cA}\sup_{\beta\in\cB}
\bigl\{-c^{\al,\beta}_{i}(x,t,p)|p|
-\tr 
\bigl(\sig^{\al,\beta}_{i}(x,t,p)(\sig^{\al,\beta}_{i})^T(x,t,p)X\bigr)\bigr\} 
\end{equation}
for $i=1,2$, 
where $\cA, \cB$ are compact metric space and 
$c^{\al,\beta}_{i}$, $\sig^{\al,\beta}_{i}$ 
are, respectively, real-valued functions 
and 
$m\times N$ matrix valued functions for some $m\in\N$ 
on $\R^N\times[0,T]\times(\R^{N}\setminus\{0\})$ 
with a possible singularity at $p=0.$ 
We assume that 
the functions $c_{i}^{\al,\beta}$, $\sig^{\al,\beta}_{i}$ 
satisfy the following conditions 
by replacing $h$ by $c_{i}^{\al,\beta}$, $\sig^{\al,\beta}_{i}$ 
for any $\al\in\cA$, $\beta\in\cB$, 
$i=1,2$, respectively{\rm :} 
$h$ are continuous on $\cQ$ and 
for some $L, M\ge0$ {\rm (}independent of $\al,\beta${\rm )}, 
\begin{equation}\label{c-1}
\begin{array}{c}
|h(x,t,p)-h(y,t,p)|\le L|x-y|, \\
|h(x,t,p)|\le M
\end{array}
\end{equation}
for all $x,y\in \R^N$, $t\in[0,T]$, $p\in\R^N\setminus\{0\}$.

Let 
$\rho, \ep\in(0,1)$, $x,y\in \ol{B}(0,R)$, 
$p=\ep^{-4}|x-y|^{2}(x-y)$,  
$X,Y\in\cS^{N}$ 
satisfy \eqref{matrix ineq} for some $R>0$ and 
let $A$ be the matrix given by \eqref{matrix-A}. 
We omit the dependence of $\al,\beta$ 
for simplicity of notation.  
We calculate that 
\begin{align*}
{}& 
(c_{1}(x,t,p)-
c_{2}(y,t,p))|p|\\
= \ &
(c_{1}(x,t,p)
-c_{1}(y,t,p))|p|
+(c_{1}(y,t,p)-
c_{2}(y,t,p))|p|\\
\le \ &
|p|(L|x-y|
+\|c_1-c_2\|_{\infty}) \\
\le \ &
L\frac{|x-y|^4}{\ep^{4}}
+M\|c_1-c_2\|_{\infty},
\end{align*}
where 
$\|\cdot\|_{\infty}
=\|\cdot\|_{\Li(B(0,R)
\times(0,T)\times(\R^N\setminus\{0\}))}$ 
and
\begin{align*}
{}&
\tr(\sig^{x}_{1}(\sig_{1}^{x})^TX)-
\tr(\sig^{y}_{2}(\sig_{2}^{y})^TY) \\
= \ &
\sum_{i=1}^{N}\bigl\{
\langle X\sig^{x}_{1}e_{i}, \sig^{x}_{1}e_{i}\rangle
-
\langle Y\sig^{y}_{2}e_{i}, \sig^{y}_{2}e_{i}\rangle
\bigr\}\\
{}\le& \ 
\sum_{i=1}^{N}\bigl\{
\big\langle 
A
\left(
\begin{array}{c}
\sig^{x}_{1}e_{i} \\
\sig^{y}_{2}e_{i} 
\end{array}
\right)
, 
\left(
\begin{array}{c}
\sig^{x}_{1}e_{i} \\
\sig^{y}_{2}e_{i} 
\end{array}
\right)
\big\rangle 
+\rho 
\big\langle 
A^{2} 
\left(
\begin{array}{c}
\sig^{x}_{1}e_{i} \\
\sig^{y}_{2}e_{i} 
\end{array}
\right)
, 
\left(
\begin{array}{c}
\sig^{x}_{1}e_{i} \\
\sig^{y}_{2}e_{i} 
\end{array}
\right) 
\big\rangle 
\bigr\}\\
{}\le
& \ 
\sum_{i=1}^{N}\bigl\{
\big\langle 
A
\left(
\begin{array}{c}
\sig^{x}_{1}e_{i} \\
\sig^{y}_{2}e_{i} 
\end{array}
\right)
, 
\left(
\begin{array}{c}
\sig^{x}_{1}e_{i} \\
\sig^{y}_{2}e_{i} 
\end{array}
\right)
\big\rangle 
+
C\rho\|A^{2}\|
(\|\sig^{x}_{1}\|^{2}_{\infty}+\|\sig^{x}_{2}\|^{2}_{\infty})
\end{align*}
for some $C>0$, 
where 
$\{e_{i}\}_{i}$ is the canonical basis of $\R^N$, 
$\sig^{x}_{1}:=
\sig_1^{\al,\beta}(x,t,p)$ and 
$\sig^{y}_{2}:=
\sig_2^{\al,\beta}(y,t,p)$.  
Due to \eqref{matrix ineq}, we have 
\begin{align*}
{}&
\big\langle 
A
\left(
\begin{array}{c}
\sig^{x}_{1}e_{i} \\
\sig^{y}_{2}e_{i} 
\end{array}
\right)
, 
\left(
\begin{array}{c}
\sig^{x}_{1}e_{i} \\
\sig^{y}_{2}e_{i} 
\end{array}
\right)
\big\rangle 
\le 
\frac{2}{\ep^{4}}|x-y|^2
|(\sig^{x}_{1}-\sig^{y}_{2})e_i|^2\\
\le \ & 
\frac{4}{\ep^4}|x-y|^2
\bigl\{
|\sig_1(x,t,p)-
\sig_1(y,t,p)|^2
+\|\sig_1-\sig_2\|_{\infty}^2
\bigr\}\\
\le \ & 
\frac{4}{\ep^4}|x-y|^2
\bigl(L^2|x-y|^2
+
\|\sig_1-\sig_2\|_{\infty}^2
\bigr). 
\end{align*}

From the above computations, 
it follows that 
the inequality (A4) holds 
by replacing $\kap_{1}$ and $\kap_{2}$ by 
$\sup_{\cA\times\cB}
\|c_1-c_2\|_{\infty}$ and 
$\sup_{\cA\times\cB}
\|\sig_1-\sig_2\|_{\infty}^{2}$, respectively. 
Therefore, 
if the $u_i$'s are solutions of \eqref{HJ} with $H=H_{i}$
given by \eqref{Hi-ex1}, $i=1,2,$ then the conclusion
of Theorem {\rm \ref{continuous dependence}} holds and reads
\begin{multline*}
\sup_{x\in\R^{N}}(u_1-u_2)(x,t)\le 
\sup_{x\in\R^{N}}(u_1-u_2)(x,0)\\
+M_{1}\sup_{\cA\times\cB}
\bigl(t\|c_1^{\al,\beta}-c_2^{\al,\beta}\|_{\infty}
+\sqrt{t}\|\sig_1^{\al,\beta}-\sig_2^{\al,\beta}\|_{\infty}
\bigr)
\end{multline*}
for all $t\in[0,T].$
Finally, note that, applying Theorem {\rm \ref{continuous dependence}}
with $H_1=H_2,$ gives comparison and uniqueness for  \eqref{HJ}.
}
\end{ex}

%

\begin{rem}\label{appr-arg} For the applications we have in mind, 
a continuous dependence result for equations with a measurable dependence 
in time will be needed. 
We do not state a precise result here but we mention that it can 
be obtained by an easy approximation argument.
\end{rem}

\section{Estimates on Lower-Bound Gradient and Properties of Fronts}

We consider the initial value problem in this section 
\begin{equation}\label{e-2}
\left\{
\begin{aligned}
&
u_{t}+H(x,t,Du,D^{2}u)=0 
&& \textrm{in} \ \Q, \\
&u(\cdot,0)=u_0
&& \textrm{in} \ \R^{N}. 
\end{aligned}
\right.
\end{equation}

We make the following assumption 
on $u_{0}$ throughout this section. 
\begin{itemize}
\item[(I1)]
$u_0\in\W(\R^N)$ and 
$|u_0(x)|\le 1$ for all $x\in\R^N$ and 
there exists $R_0>0$ such that $u_0(x)=-1$ 
for all $x\in\R^N\setminus B(0,R_0)$. 

\item[(I2)]
There exist constants $\lam_0, \del_0\in(0,1)$, 
$\eta_0>0$ and $\nu\in C(\R^N,\R^{N})$ such that 
\begin{eqnarray} \label{inegI2}
u_0(x+\lam\nu(x))\ge u_0(x)+\lam\eta_{0}
\quad\textrm{for all} 
\ x\in U_{0}, \ \lam\in[0,\lam_{0}], 
\end{eqnarray}
where $U_{0}:=\{x\in\R^{N}\mid |u_0(x)|\le \del_{0}\}$. 
\end{itemize}

\begin{rem}\label{rmk-C1-lisse-geom}
Without loss of generality, we may assume that 
$\nu$ is 
a smooth bounded Lipschitz continuous function 
and henceforth 
we will assume it from now on. 
Indeed, let $\nu_{\ep}\in C^{\infty}(\R^{N},\R^{N})$ for $\ep\in(0,1)$ 
be an approximate function of $\nu$, then we have 
\begin{align*}
u_{0}(x+\lam\nu_{\ep}(x))
&=\, 
u_{0}(x+\lam\nu(x)+\lam(\nu_{\ep}(x)-\nu(x)))\\
{}&\ge\, 
u_{0}(x+\lam\nu(x))
-\lam\|Du_{0}\|_{L^{\infty}(U_0)}\|\nu_{\ep}-\nu\|_{L^{\infty}(U_0)}. 
\end{align*}
If $\ep$ is enough small, then we have 
\[
u_0(x+\lam\nu_{\ep}(x))\ge u_0(x)+\lam\frac{\eta_{0}}{2}
\quad\textrm{for all} 
\ x\in U_{0}, \ \lam\in[0,\lam_{0}]. 
\]
\end{rem}

 Let $u_0\in C^1(\R^N)$ such that 
\begin{equation}\label{cond641}
Du_0\not= 0 \quad {\rm on} \quad \Gamma_0:=\{u_0=0\}.
\end{equation}
Then, for $\del_{0}>0$ and $\lam_{0}$ enough small, 
$Du_0\not= 0$ in $U_0+B(0,\lam_{0}\|\nu\|_{\infty})$ 
and, setting $\nu(x)=Du_0(x),$ we have 
\begin{equation*}
u_{0}(x+\lam\nu(x))= 
u_{0}(x)+\lam |Du_{0}(x)|^2+
 \lam \omega_{U_0}(M\lam ),
\end{equation*}
where $\om_{U_0}$ is a modulus of continuity of $Du_0$
in $\ol{U}_0+B(0,\lam_{0}\|\nu\|_{\infty})$ 
and $M=\max_{\ol{U}_0+B(0,\lam_{0}\|\nu\|_{\infty})}|Du_0|.$
Therefore (I2) holds for 
$\eta = \min_{\ol{U}_0+B(0,\lam_{0}\|\nu\|_{\infty})}|Du_0|/2$
and $\lam\in [0,\lam_0]$ for $\lam_0$ enough small.
Moreover, the Implicit 
Function Theorem implies that $\Gamma_0$ is a $C^1$ hypersurface.
Conversely, assume that $\Gamma_0$ is a $C^1$ hypersurface
with the {\it unique nearest point property} (that is,
there exists a neighborhood $U_0$ of $\Gamma_0$ such that,
for all $x\in U_0,$ there exists a unique 
$\overline{x}\in \Gamma_0$ such that ${\rm dist}(x,\Gamma_0)
= |x-\overline{x}|$). Then the signed distance function $d_{\Gamma_0}^s$ 
to $\Gamma_0$ is $C^1$ (see \cite{foote84}). It follows that
(I1), (I2) hold with
$u_0$ such that $u_0= d_{\Gamma_0}^s$ in a neighborhood of $\Gamma_0$
and $u_0$ is a suitable regularization of $d_{\Gamma_0}^s$ elsewhere. 
More generally, when
considering front propagation problems, it may be convenient to have
a characterization of (I2) in geometrical terms.
Such a result does not seem obvious. However, we have partial
results in the following lemma, the proof of which is given in the appendix
with additional comments. 

A subset $A\subset\R^N$ is {\it star-shaped} with respect to $x_0$ if,
for every $x\in A,$ the segment 
$[x_0,x):=\{\lam x+(1-\lam)x_0, \lam\in [0,1)\}$
belongs to $A.$ It is {\it star-shaped with respect to a ball}
$B(x_0,r_0)$ if $A$ is star-shaped with respect to every $y\in B(x_0,r_0).$

\begin{lem} \label{I2geom}
Let $\Omega_0\subset \R^N$ be an open bounded set with boundary
$\partial\Omega_0=:\Gamma_0.$

\begin{itemize}
\item[(i)] 
{\rm (Star-shaped with respect to a ball domains)}
The set $\Omega_0$ is star-shaped with respect to a ball, i.e.,
there exists a compact subset ${\mathcal K}\subset\R^N$ and $r_0>0$
such that
\begin{equation}\label{strict-etoile-boule}
\Omega_0= \mathop{\bigcup}_{x\in {\mathcal K}}\mathop{\bigcup}_{\alpha\in [0,1]}
\ol{B}(\alpha x, (1-\alpha)r_0),
\end{equation}
if and only if there exists $u_0:\R^N\to \R$ 
such that 
\begin{equation}\label{u0representant}
\Gamma_0=\{u_0=0\}, \qquad
\Omega_0=\{u_0>0\}
\end{equation}
and {\rm (I1), (I2)} hold in $U_0$ with $\nu(x)=-x.$ 
In this case, $\Gamma_0$
is  locally the graph of a Lipschitz continuous
function.

\item[(ii)]
If there exists $K>0$ such that 
$\Gamma_0$ is locally the graph of a Lipschitz continuous
function with constant $K,$ then there exists $u_0$ such that
\eqref{u0representant},  {\rm (I1)} and {\rm (I2)} hold.
\end{itemize}
\end{lem}

Hereinafter, we set 
\begin{equation}\label{defpsi}
\psi_{\lam}(x):=x+\lam\nu(x)= (I+\lam\nu)(x).
\end{equation} 
From Remark \ref{rmk-C1-lisse-geom},
we may assume that $\nu$ is a smooth 
bounded Lipschitz continuous function and,
replacing $\lam_0$ by a smaller constant in order that
\begin{equation}\label{deflambda0}
\lam_0 \|\nu\|_\infty < 1 \quad {\rm and} \quad  
\lam_0 \|D\nu\|_\infty < 1,
\end{equation} 
we obtain that $\psi_{\lam}(x)$ is 
a $C^1$-diffeomorphism in $\R^N$ with
\begin{equation} \label{defxilam}
\xi_{\lam}:= \psi_{\lam}^{-1}= (I+\lam\nu)^{-1}, \quad
D\xi_{\lam}(x)= I+\sum_{k=1}^{\infty}(-\lam D\nu(x))^{k}.
\end{equation}

We assume (A1)--(A3) and 
make the following additional assumptions 
on $H$ throughout this section. 
\begin{itemize}
\item[{\rm (A5)}] 
The function $H$ is \textit{geometric}, i.e., 
\[
H(x,t,\al p, \al X+\beta p\otimes p)
=\al H(x,t,p,X)
\] 
for all $\al>0$, $\beta\in\R$, 
$(x,t,p,X)\in\cQ\times(\R^{N}\setminus\{0\})\times\cS^{N}$.

\item[{\rm (A6)}]
There exists $L_{H}>0$ such that 
\[|H(x,t,p,X)-H(x,t,p,Y)|\le L_{H}|X-Y|\]
for any $(x,t,p)\in\cQ\times(\R^{N}\setminus\{0\})$ 
and $X,Y\in\cS^{N}$. 

\item[{\rm (A7)}]
For any $R>0$, there exists $C_{H}>0$ such that 
\begin{align}
{}&
H\bigl(\psi_{\lam}(y),t,D\xi_{\lam}(\psi_{\lam}(y))^Tp, 
D\xi_{\lam}(\psi_{\lam}(y))^{T}YD\xi_{\lam}(\psi_{\lam}(y))\bigr) 
-
H(x,t,p, X) \nonumber\\
\le \, & 
C_{H}
\bigl(\frac{|x-y|^{4}}{\ep^{4}}+ 
\lam+\frac{\lam^{2}|x-y|^{2}}{\ep^{4}}+\rho\|A^{2}\|\bigr) 
\label{lower-assump}
\end{align}
for any $\lam\in[0,\lam_{0}]$ 
and 
for any $\rho,\ep\in(0,1)$, 
$x,y\in \ol{B}(0,R)$, 
$t\in[0,T]$, 
$p=4\ep^{-4}|x-y|^{2}(x-y)$, 
$X,Y\in\cS^{N}$ satisfying 
\eqref{matrix ineq} and 
$A\in\cS^{N}$ given by~\eqref{matrix-A}.  

\item[{\rm (A8)}] There exists at least one
viscosity solutions of \eqref{e-2} which
satisfies 
\begin{equation}\label{condmoins1}
u(x,t)=-1 
\quad\textrm{for all} \ 
(x,t)\in(\R^N\setminus B(0,R_{T}))\times[0,T] 
\end{equation}
for some $R_{T}>0$. 
\end{itemize}

Let us make some comments about these new assumptions:
(A5) is needed to use the level set approach to describe
front propagation (see \cite{BSS, G} for instance). Assumption (A6)
is satisfied for a wide class of quasilinear equations
under interest in this paper, see Example \ref{ex-2}.
A consequence of (A5) and (A6) is: 
For any $R>0$, there exists $M_{R}>0$ such that 
\begin{equation}\label{ancien(A6)}
|H(x,t,p,X)|\le M_{R}(1+|X|) 
\quad\textrm{on} \ 
\ol{B}(0,{R})\times[0,T]\times 
({B}(0,R)\setminus \{0\})
\times\cS^{N}, 
\end{equation}
which is a crucial property to obtain H\"older continuity
in time for the solutions of \eqref{e-2}, see Proposition~\ref{regularity}.
Assumption (A7) is a natural condition to obtain a
preservation of the initial property (I2) during the evolution.
This condition is related to (A4);
it is worthwhile to notice, as it was done at the end of Example \ref{ex-1}, 
that such a condition gives uniqueness for the solutions of \eqref{HJ}.
Existence of solutions to \eqref{e-2} is assumed in (A8)
because it is not the point in this paper, see \cite{GGI, SS, BCLM3} for some
conditions which guarantee existence. More precisely, we have the following
result about solutions of \eqref{e-2} and the proof is given in
 Appendix:
\begin{prop}[Regularity of Solutions]\label{regularity}
There exists a unique viscosity solution $u\in C(\cQ)$ 
of \eqref{e-2} and we have
\begin{equation}\label{reg-lip-u-x}
|u(x,t)-u(y,t)|\le\|Du_{0}\|_{\Li(\R^N)}e^{Kt}|x-y|,
\end{equation}
\begin{equation}\label{reg-hold-u-t}
|u(x,t)-u(x,s)|\le\tilde{L}|t-s|^{1/2} 
\end{equation}
for all $x,y\in\R^N$, $t,s\in[0,T]$, 
where $K, \tilde{L}$ are positive constants which depend only 
on $C_{H}$ and $C_{H}, R_{T}, \|Du_{0}\|_{\infty}$ respectively. 
\end{prop}

Now, we state the main result of this section.

\begin{thm}[Key Estimate]\label{key estimate}
There exist $\ol{t}>0,$ $0<\ol{\lam}\leq \lam_0$
{\rm (}$\lam_0$ is given by {\rm (I2)} and satisfies~\eqref{deflambda0}{\rm )} 
and a non-increasing continuous function 
$\eta:[0,\ol{t}\land T]\to[0,\infty)$ 
which depend
only on $C_H,$ $L_H,$ $ \del_0,$ 
$\eta_0,$ $R_{T},$ $\|Du_{0}\|_{\infty},$ 
$\|\nu\|_{\infty},$ $\|D\nu\|_{\infty}$ such that 
\[
\eta(t)>0 \ \textrm{for all} \ t\in[0,\ol{t}\land T), 
\]
and $u$ satisfies 
\begin{equation}\label{main-estimate}
u(x+\lam\nu(x), t)\ge 
u(x,t)+\lam\eta(t) 
\quad \textrm{for all} \ 
x\in U_{t}, \ \lam\in[0,\ol{\lam}],
\end{equation}
where 
$U_{t}:=\{x\in\R^{N}\mid 
|u(x,t)|\le \del_0/4\}$.  
\end{thm}

\begin{proof}
From (I2) and Lemma \ref{psi} in the Appendix, 
we can extend~\eqref{inegI2} in  $\R^N,$
\begin{eqnarray*} 
u_{0}(\psi_{\lam}(x))\ge \Psi(u_{0}(x)+\lam\eta_{0}) 
\quad \textrm{for any} \ x\in\R^{N}, \lam\in[0,\ol{\lam}],
\end{eqnarray*}
where $\ol{\lam}$ and $\Psi$ are introduced in Lemma \ref{psi}. 
 
We first prove 
\begin{equation}\label{key estimate lem ineq}
u(\psi_{\lam}(x),t)\ge 
\Psi(u(x,t)+\lam\eta_{0})-M_{2}\lam\sqrt{t} 
\end{equation}
for all 
$(x,t)\in\cQ$, $\lam\in[0,\ol{\lam}]$, 
some constant $M_{2}>0$, which is depends only on 
$C_H, L_H, \del_0, \eta_0, R_{T},$ $\|Du_{0}\|_{\infty}$ 
and $\|\nu\|_{\infty}$ 
(Note that $M_2$ does not depend on $\|D\nu\|_{\infty}$
contrary to $\ol{\lam}$ which depends on $\|D\nu\|_{\infty}$
through $\lam_0$ because of~\eqref{deflambda0}).

Fix $\lam\in[0,\ol{\lam}]$. 
Set 
$v(x,t):=u(\psi_{\lam}(x),t)$ and 
$w(x,t):=\Psi(u(x,t)+\lam\eta_{0})$ 
for all $(x,t)\in\cQ$. 
Since $H$ is geometric and $\Psi$ is a nondecreasing function,  
the functions $v,w$ satisfy 
\begin{numcases}
{}
v_{t}+
H\bigl(\psi_{\lam}(x),t,D\xi_{\lam}(\psi_{\lam}(x))^TDv(x,t), 
\nonumber\\
D\xi_{\lam}(\psi_{\lam}(x))^{T}D^{2}v(x,t)D\xi_{\lam}(\psi_{\lam}(x))
+D^{2}\xi_{\lam}(\psi_{\lam}(x))Dv(x,t)
\bigr) 
=0 
& in $\Q$, \nonumber\\
v(x,0)=u_0(\psi_{\lam}(x)) & in $\R^N$, \nonumber
\end{numcases}
\begin{numcases}
{}
w_{t}+H(x,t,Dw,D^{2}w)=0 
& in $\Q$, \nonumber\\
w(x,0)=\Psi(u_0(x)+\lam\eta_{0})  & in $\R^N$ \nonumber
\end{numcases}
in the viscosity sense 
(see \cite[Theorem 4.2.1]{G} 
for instance).

Let $R_{T}$ be the constant in (A8) 
and recall that $\ol{\lam}\eta_{0}\le \del_{0}/4$ 
in Lemma~\ref{psi}. 
For any $(x,t)\in(\R^N\setminus \ol{B}(0,R_{T}+\ol{\lam}\|\nu\|_{\infty})\times[0,T]$, 
$u(x,t)+\lam\eta_{0}\le -1+\del_{0}/4\le -(3\del_{0})/4$, 
which implies that 
$-1=\Psi(u(x,t)+\lam\eta_{0})\le u(\psi_{\lam}(x),t)$. 
Therefore, we only need to show that 
for any $(x,t)\in \ol{B}(0,R_{T}+\ol{\lam}\|\nu\|_{\infty})\times[0,T]$ 
inequality \eqref{key estimate lem ineq} holds. 
Note that 
\[
|D^{2}\xi_{\lam}(\psi_{\lam}(x))p|
\le
C\lam|p|. 
\]
By Assumptions (A6), (A7) and Theorem \ref{continuous dependence} 
with $\kap_{1}=\lam$ and $\kap_{2}=\lam^{2}$,
we get, for any $(x,t)\in \ol{B}(0,R_{T}+\ol{\lam}
\|\nu\|_\infty )\times[0,T]$, 
\[
(w-v)(x,t)
\le 
C(t+\sqrt{t})\lam
\le 
C(\sqrt{T}+1)\sqrt{t}\lam=:M_{2}\sqrt{t}\lam 
\]
for all $t\in[0,T]$, 
which implies \eqref{key estimate lem ineq}.

Setting 
$\ol{t}:=(\eta_{0}/M_{2})^2$ and 
$\eta(t):=\eta_{0}-M_{2}\sqrt{t}$ for all $t\in[0,\ol{t}\land T]$, 
we obtain the conclusion. 
\end{proof}

The first important consequence is a lower-gradient
bound estimate on the front. 

\begin{cor}[Estimate on Lower-Bound Gradient]\label{lower bound}
We have
\[
-|Du(x,t)|\le -\frac{\eta(t)}{\|\nu\|_{\infty}} 
\quad\textrm{in} \ 
\{|u(\cdot,t)|<\frac{\del_{0}}{4}\}\times(0,\ol{t}\land T), 
\]
where $\ol{t}$ and $\eta$ are given in Theorem~{\rm \ref{key estimate}}. 
\end{cor}

\begin{rem}
Theorem 3.1 in \cite{BP} implies that 
we cannot expect global in time lower gradient estimates 
for solutions of \eqref{m-2} with general initial data 
like (I1), (I2), even if we assume some positiveness
assumptions on the velocity like in \cite{ACM, BL, BCLM1, BCLM2}.
\end{rem}

Before giving the proof of this result,
we continue by stating another consequences of
Theorem~\ref{key estimate}. We need to introduce some notations. 

For any $t\in[0,T]$, $r\in[-1,1]$, we set 
\[
\Om_t^{r}:=\{x\in\R^N\mid u(x,t)>r\}, \quad 
\Gam_t^{r}:=\pl\Om_t^{r} 
\]
and define the cone with 
vertex $z\in\R^N$, axis $e\in\bS^{N-1}$ and 
parameters $(\rho,\theta)\in\R_{+}\times\R_{+}$ by
\begin{align*}
C^{\rho,\theta}_{e,z}
&:= \, 
\bigcup_{a\in[0,\theta]} \ol{B}(z+ae,a\,\frac{\rho}{\theta}) \\ 
{}&= \, 
\{z+ae+a\,\frac{\rho}{\theta}\xi
\mid a\in[0,\theta],\xi\in\ol{B}(0,1)\}. 
\end{align*}

The following result means that the evoluting fronts have
the interior cone property. 

\begin{cor}[Interior Cone Properties of Fronts]\label{cone}
For any $r\in[-\del_{0}/4,\del_{0}/4]$ and 
$t\in[0,\ol{t}\land T]$, 
\[
C^{\rho(t),\theta(z)}_{\frac{\nu(z)}{|\nu(z)|},z}
\subset\ol{\Om}_{t}^{r} 
\quad 
\textrm{for all} \ 
z\in\Gam_{t}^{r}, 
\]
where 
\[
\rho(t):=\frac{\eta(t)\ol{\lam}}{\|Du_{0}\|_{\infty}e^{Kt}},
\quad
\theta(z):=\ol{\lam}|\nu(z)|.  
\]
\end{cor}

When $A$ is a subset of $\R^k,$ we will write, by abuse of notation,
${\rm Per}(A)=\cH^{k-1}(\partial A)$ for the perimeter of $A$.
Notice that it does not always correspond to the usual definition
of perimeter. The two definitions coincide for instance when the 
boundary is locally
the graph of a Lipschitz function, which is often the case in our 
applications.
For further details, see
\cite[Section 5 and Remark p.183]{EG} or~\cite{hp05}. 

\begin{cor}[Estimate on Perimeter of Fronts]\label{estimate perimeter}
There exists a constant $M_{3}>0$ which depends 
only on the constants appearing in Theorem~{\rm \ref{key estimate}} 
such that 
\[
\Per(\Om^{r}_{t})\le M_{3}
\]
for all 
$r\in[-\del_{0}/4,\del_{0}/4]$ and 
$t\in[0,(\ol{t}\land T)/2].$ 
\end{cor}

We turn to the proofs.

\begin{proof}[Proof of Corollary {\rm \ref{lower bound}}]
Take any function $\phi\in C^{1}(\Q)$ satisfying 
$(u-\phi)(x_{0},t_{0})=0$ and 
$(u-\phi)(x,t)\le 0$ for all $(x,t)\in\Q$ 
for some $(x_{0},t_{0})\in 
\{|u(\cdot,t)|<\del_{0}/4\}\times(0,\ol{t}\land T)$,
where $\ol{t}$ is given by Theorem~\ref{key estimate}. 
By Theorem \ref{key estimate} and mean-value theorem, 
we have 
\begin{align*}
\lam\eta(t_{0})
\le \ &
u(\psi_{\lam}(x_{0}),t_{0})-u(x_{0},t_{0})\\
{}\le \ &
\phi(\psi_{\lam}(x_{0}),t_{0})-\phi(x_{0},t_{0})\\
{}= \ &
\lam \langle D\phi(x_{0},t_{0}),\nu(x)\rangle+o(\lam)\\
{}\le \ &
\lam|D\phi(x_{0},t_{0})|\|\nu\|_{\infty}+o(\lam\|\nu\|_{\infty}).
\end{align*}
Dividing by $\lam\|\nu\|_{\infty}>0$ in the above and 
letting $\lam\in(0,\ol{\lam}]$ go to $0$, 
we obtain the conclusion. 
\end{proof}

\begin{proof}[Proof of Corollary {\rm \ref{cone}}]
Fix $r\in[-\del_{0}/4,\del_{0}/4]$, 
$t\in[0,\ol{t}\land T]$ and $z\in\Gam_{t}^{r}$. 
By Theorem \ref{key estimate}, we have
\[
u(z+\lam\nu(z),t)\ge r+\lam\eta(t) 
\quad\textrm{for all} \ 
\lam\in[0,\ol{\lam}].  
\]
Set 
$r_{\lam}(t):=(\lam\eta(t))/(\|Du_{0}\|_{\infty}e^{Kt})$. 
For any $\xi\in\ol{B}(0,1)$, we have 
\begin{align*}
u(z+\lam\nu(z)+r_{\lam}(t)\xi,t)
&\ge\,
u(z+\lam\nu(z),t)-\|Du_{0}\|_{\infty}e^{Kt}r_{\lam}(t)\\
{}&\ge
r+\lam\eta(t)-\|Du_{0}\|_{\infty}e^{Kt}r_{\lam}(t)
\ge
r, 
\end{align*}
which implies that 
\[
\ol{B}(z+\lam\nu(z),r_{\lam}(t))\subset
\ol{\Om}_{t}^{r}
\]
for any $\lam\in[0,\ol{\lam}]$. 
Therefore, we have 
\[
C^{\rho(t),\theta(z)}_{\frac{\nu(z)}{|\nu(z)|},z}
=
\bigcup_{\lam\in[0,\ol{\lam}]}
\ol{B}(z+\lam\nu(z), \lam\frac{\rho(t)}{\ol{\lam}})
=
\bigcup_{a\in[0,\theta(z)]}
\ol{B}(z+a\frac{\nu(z)}{|\nu(z)|}, a\frac{\rho(t)}{\theta(z)})
\subset
\ol{\Om}_{t}^{r}. 
\]
\end{proof}

Before doing 
the proof of Corollary~\ref{estimate perimeter}, 
we recall the following lemma. 

\begin{lem}[{\cite[Theorem 5.8]{BCLM2}}]\label{perimeter}
Let $K$ be a compact subset of $\R^N$ having the interior cone 
property of parameters $\rho$ and $\theta$. 
Then there exists a positive constant 
$\Lam=\Lam(N,\rho,\theta/\rho)$ 
such that for all $R>0$, 
\[
\cH^{N-1}(\pl K\cap\ol{B}(0,R))\le\Lam\cL^{N}(K\cap\ol{B}(0,R+\rho/4)). 
\]
\end{lem}

\begin{proof}[Proof of Corollary {\rm \ref{estimate perimeter}}]
Set $t^{\ast}:=(\ol{t}\land T)/2$. 
Let $\rho(t), \theta(z)$ be the functions in 
Corollary~\ref{cone} and 
set $\ol{\rho}:=\rho(t^{\ast})$ and 
$\ol{\theta}:=\min_{z\in\pl\Om_{t},t\in[0,t^{\ast}]}
\theta(z)$.
By Theorem~\ref{key estimate} and Corollary~\ref{cone}, 
we see that $\ol{\rho}, \ol{\theta}>0$ and we have,
\[
C^{\ol{\rho},\ol{\lam}}_{\frac{\nu(z)}{|\nu(z)|},z}
\subset
C^{\rho (t^{\ast}),\theta(z)}_{\frac{\nu(z)}{|\nu(z)|},z}
\subset
\ol{\Om}^{r}_{t} 
\quad\textrm{for all} \ 
z\in\Gam^{r}_{t}, 
r\in[-\frac{\del_{0}}{4},\frac{\del_{0}}{4}]. 
\]

Due to Lemma \ref{perimeter}, 
there exists $\Lam=\Lam(N,\ol{\rho},
\ol{\theta}/\ol{\lam}
)>0$ 
such that 
\[
\cH^{N-1}(\Gam^{r}_{t})
\le 
\Lam\cL^{N}(\ol{\Om}^{r}_{t})
\le 
\Lam \cL^{N}(B(0,R_{T}))=:M_3 
\]
for all $t\in[0,t^{\ast}]$, 
$r\in[-\del_{0}/4,\del_{0}/4]$. 
\end{proof}

We end this section with an application.

\begin{ex}\label{ex-2}
{\rm 
We consider the function 
\begin{equation}\label{example equation}
H(x,t,p,X)=
\inf_{\al\in\cA}\sup_{\beta\in\cB}
\bigl\{-c^{\al,\beta}(x,t,p)|p|
-\tr \bigl(\sig^{\al,\beta}(x,t,p)(\sig^{\al,\beta})^T(x,t,p)X\bigr)\bigr\}, 
\end{equation}
where 
the functions 
$c^{\al,\beta}$ and $\sig^{\al,\beta}$ satisfy 
\eqref{c-1} for all $\al\in\cA$, $\beta\in\cB$, 
respectively. 
We add the following assumptions on 
$c^{\al,\beta}, \sig^{\al,\beta}${\rm :} 
\begin{gather}
c(x,t,\mu p)=c(x,t,p), \quad
|c(x,t,p)-c(x,t,q)|\le C_{c}|p-q|, \nonumber \\
\sig(x,t,\mu p)=\sig(x,t,p), \quad 
\sig^{T}(x,t,p)p=0, \quad 
|\sig(x,t,p)-\sig(x,t,q)|\le \frac{C_{\sig}|p-q|}{|p|+|q|}
\label{siggeom}
\end{gather}
for all $\mu>0$, $(x,t)\in\cQ$, 
$p,q\in(\R^N\setminus\{0\})$ 
and some $C_{c}, C_{\sig}>0$. 
These assumptions are related to (A5). 
A typical example is 
$\sig^{\al,\beta}(x,t,p)=I-p\otimes p/|p|^2$ 
and then the second-order term is the so-called mean curvature term.
We claim that 
the function $H$ satisfies {\rm (A1)--(A8)}.

It is easy to check that 
the function $H$ satisfies 
{\rm (A1)--(A5)}. 
We check that the function $H$ satisfies (A7). 
Note that, 
by~\eqref{deflambda0},~\eqref{defxilam}, we have
\begin{gather*}
|D\xi_{\lam}(\psi_{\lam}(x))|
\le
\frac{1}{1-\lam|D\nu(x)|}
\le
\frac{1}{1-\lam_0 \|D\nu\|_\infty}
<+\infty, \\
|I-D\xi_{\lam}(\psi_{\lam}(x))|
\le
\frac{\lam \|D\nu\|_\infty }{1-\lam_0 \|D\nu\|_\infty},
\end{gather*}
for $\lam_{0}\in(0,1)$ small enough and any $x\in\R^{N}$. 
By abuse of notations, we write $c, \sig$ instead of 
$c^{\al,\beta}, \sig^{\al,\beta}$ for any $\al\in\cA, \beta\in\cB$. 
We compute 
\begin{align*}
{}&
|c(x,t,p)-c(\psi_{\lam}(x),t,D\xi_{\lam}(\psi_{\lam}(x))^{T}p)|\\
\le&\, 
|c(x,t,p)-c(\psi_{\lam}(x),t,p)|
+|c(\psi_{\lam}(x),t,p)-c(\psi_{\lam}(x),t,D\xi_{\lam}(\psi_{\lam}(x))^{T}p)|\\
\le&\, 
L|x-\psi_{\lam}(x)|
+
C_{c}|(I-D\xi_{\lam}(\psi_{\lam}(x))^{T})p|\\
\le&\, 
\tilde{C}\lam 
\end{align*}
and 
\begin{align*}
{}&
|\sig(x,t,p)-
D\xi_{\lam}(\psi_{\lam}(x))\sig(\psi_{\lam}(x),t,
D\xi_{\lam}(\psi_{\lam}(x))^{T}p)|\\ 
\le&\, 
|\sig(x,t,p)-D\xi_{\lam}(\psi_{\lam}(x))\sig(x,t,p)|
+
|D\xi_{\lam}(\psi_{\lam}(x))(\sig(x,t,p)-\sig(\psi_{\lam}(x),t,p))|\\
&\, 
+
|D\xi_{\lam}(\psi_{\lam}(x))(\sig(\psi_{\lam}(x),t,p)-
\sig(\psi_{\lam}(x),t,D\xi_{\lam}(\psi_{\lam}(x))^{T}p))|\\
\le&\, 
M|I-D\xi_{\lam}(\psi_{\lam}(x))|
+L|D\xi_{\lam}(\psi_{\lam}(x))||x-\psi_{\lam}(x)|\\
&\, 
+C_{\sig}\frac{|D\xi_{\lam}(\psi_{\lam}(x))||(I-D\xi_{\lam}(\psi_{\lam}(x))^{T})p|}
{|p|+|D\xi_{\lam}(\psi_{\lam}(x))^{T}p|}\\
\le&\, 
\tilde{C}\lam 
\end{align*}
for some 
$\tilde{C}=\tilde{C}(L,M,C_{c},C_{\sig})>0$ 
and any $x\in\R^{N}, p\in\R^{N}\setminus\{0\}$ 
and $\lam\in[0,\lam_{0}]$.

By using the same computations as Example {\rm \ref{ex-1}}, 
we have 
\begin{align*}
{}&\, 
H\bigl(\psi_{\lam}(y),t,D\xi_{\lam}(\psi_{\lam}(y))p, 
D\xi_{\lam}(\psi_{\lam}(y))^{T}YD\xi_{\lam}(\psi_{\lam}(y))\bigr) 
-H(x,t,p,X)\\
\le&\, 
C\bigl(\frac{|x-y|^{4}}{\ep^{4}}+ 
\lam+\frac{\lam^{2}|x-y|^{2}}{\ep^{4}}+\rho\|A^{2}\|\bigr) 
\end{align*}
for some $C=C(L,M,C_{c},C_{\sig})>0$ 
and any $\rho,\ep\in(0,1)$, 
$x,y\in \ol{B}(0,R)$, 
$t\in[0,T]$, 
$p=\ep^{-4}|x-y|^{2}(x-y)$ with $x\not=y$, 
$X,Y\in\cS^{N}$ satisfying \eqref{matrix ineq} 
and 
$A\in\cS^{N}$ given by \eqref{matrix-A}, 
which implies that $H$ satisfies (A7).

We finally check that $H$ satisfies 
{\rm (A8)}. 
At first, the constant function
$-1$ is obviously a subsolution of \eqref{e-2}. 
We set $R(t):=M_{c}t+R_{0}+\sqrt{2},$ 
with $M_{c}\geq \|c\|_\infty$ and 
define the function $f:\cQ\to\R$ by 
\[
f(x,t):=
\phi((R(t)-|x|)^2-1), 
\]
where $\phi(r):=r\vee (-1)$. 
We prove that $f$ is a viscosity supersolution 
of \eqref{e-2}. 
It is easily seen that $f(x,0)\ge u_{0}(x)$ on $\R^N$. 
Indeed, 
for all $x\in B(0,R_{0})$, we have 
$f(x,0)\ge 
(R_{0}+\sqrt{2}-|x|)^{2}-1\ge 1\ge u_{0}(x)$ 
and, for all $x\in \R^N\setminus B(0,R_{0})$, 
$u_0(x)=-1\le f(x,0)$ (see (I1)).
We have 
$f_t(x,t)=2M_{c}(R(t)-|x|)$, 
$Df(x,t)=2(|x|-R(t))\frac{x}{|x|}$ and 
$D^{2}f(x,t)=
2\bigl(I-\frac{R(t)}{|x|}(I-\frac{x\otimes x}{|x|^2})\bigr)$ 
for any $t\in(0,T)$ and 
$x\in \ol{B}(0,R(t))\setminus\{0\}$. 
Note that $D^{-}f(0,t)=\emptyset$ for all $t\in[0,T]$. 
Set $b^{\al,\beta}(x,t,p)=\sig^{\al,\beta}(x,t,p)(\sig^{\al,\beta})^T(x,t,p)$. 
We calculate that 
\begin{align*}
{}&
f_{t}
+\inf_{\al\in\cA}\sup_{\beta\in\cB}
\{-c^{\al,\beta}(x,t,Df)|Df|-\tr(b^{\al,\beta}(x,t,Df)D^{2}f)\}\\
=&
\inf_{\al\in\cA}\sup_{\beta\in\cB}\Bigl\{
2(M_c-c^{\al,\beta}(x,t,Df))(R(t)-|x|)\\
{}&-2\tr\Bigl(b^{\al,\beta}(x,t,Df)
\bigl(I-\frac{R(t)}{|x|}(I-\frac{x\otimes x}{|x|^2})\bigr)\Bigr)\Bigr\}. 
\end{align*}
Set $e_{1}:=x/|x|$ and 
take $e_{i}\in\R^N$ for $i=2,\ldots,N$ so that 
$\{e_{i}\}_{i=1,\ldots,N}$ is an orthonormal basis. 
Then we have 
$(I-\frac{x\otimes x}{|x|^2})e_{1}=0$ and 
$(I-\frac{x\otimes x}{|x|^2})e_{i}=e_{i}$ 
for $i=2,\ldots,N$. 
Therefore, 
\begin{align}
{} & 
\tr\Bigl(b^{\al,\beta}(x,t,Df)
\bigl(I-\frac{R(t)}{|x|}(I-\frac{x\otimes x}{|x|^2})\bigr)\Bigr)\nonumber\\
= \ & 
\sum_{i=1}^{N}
\big\langle b^{\al,\beta}(x,t,Df)
\bigl(I-\frac{R(t)}{|x|}(I-\frac{x\otimes x}{|x|^2})\bigr)e_{i},e_{i}
\big\rangle)\nonumber\\
= \ & 
\sum_{i=1}^{N}
\big\langle b^{\al,\beta}(x,t,Df)e_{i},e_{i}\big\rangle 
-\frac{R(t)}{|x|}
\big\langle b^{\al,\beta}(x,t,Df)\bigl(I-\frac{x\otimes x}{|x|^2}\bigr)
e_{i},e_{i}\big\rangle)\nonumber \\
= \ & 
\sum_{i=2}^{N}
\bigl(1-\frac{R(t)}{|x|}\bigr)\langle 
(\sig^{\al,\beta})^T(x,t,Df)e_{i},(\sig^{\al,\beta})^T(x,t,Df)e_{i}\rangle)
\nonumber\\
\le \ & 
0, \label{ineg564}
\end{align}
since $|x|<R(t)$ and
$(\sig^{\al,\beta})^T(x,t,Df)e_1=(\sig^{\al,\beta})^T(x,t,-x)x=0$
by \eqref{siggeom}. Moreover, $D^{-}u(0,t)=\emptyset$ and $-1$ is
obviously a supersolution on $\R^N\setminus B(0,R(t))$. 
Setting $R_{T}=R(T)$, 
we see that {\rm (A8)} is satisfied
in view of the comparison theorem for viscosity solutions 
of \eqref{e-2}.  
}
\end{ex}


\section{Uniqueness of Solutions of Nonlocal Equations}

In this section, 
we consider the initial value problem of 
the nonlocal and non-monotone 
geometric equations which is derived from \eqref{general law}, 
through the \textit{level set approach} (see \cite{CGG,ES,G}), 
\begin{equation}\label{e-1}
\left\{
\begin{aligned}
&
u_t+
H[\mathbf{1}_{\{u\ge0\}}](x,t,Du,D^{2}u)=0 
&& \textrm{in} \ \Q, \\
&u(\cdot,0)=u_0
&& \textrm{in} \ \R^{N}.  
\end{aligned}
\right.
\end{equation}
For any function 
$\chi\in\Li(\cQ,[0,1])$,  
$H[\chi]$  
denotes a real-valued function of 
$(x,t,p,X)\in\R^N\times[0,T]\times(\R^N\setminus\{0\})\times\cS^{N}$. 
For almost any $t\in [0, T]$, 
$(x,p,X)\mapsto H[\chi](x,t,p,X)$ 
are continuous functions 
on $\R^N\times(\R^N\setminus{\{0\}})\times\cS^{N}$ 
with a possible singularity 
at $p=0$. 
For all $(x,p,X)\in\R^{N}\times(\R^N\setminus\{0\})\times\cS^{N}$, 
$t\mapsto H[\chi](x,t,p,X)$ 
are measurable functions. 
For any $\chi\in\Li(\cQ,[0,1])$, 
$H[\chi]=H$ satisfies (A2), (A3), (A5).

Furthermore, we make the following assumptions 
(H1)--(H5-(i)) or (H5-(ii)) and (I1), (I2) on $u_{0}$ 
throughout this section.

\begin{itemize}
\item[(H1)]
For any $\chi\in\Li(\cQ,[0,1])$, 
equation \eqref{e-1} has a bounded uniformly continuous 
$L^{1}$-viscosity solution $u[\chi]$. 
Moreover, there exist constants $C,R_{T}>0$ 
independent of $\chi\in\Li(\cQ,[0,1])$ 
such that $|u(x,t)|\le C$ for all $(x,t)\in\cQ$ 
and $u(x,t)=-1$ for all 
$x\in (\R^N\setminus B(0,R_{T}))\times[0,T]$.

\item[(H2)]
For any $\tau\in[0,T]$ and $\chi\in C([0,\tau];L^{1}(\R^N))$ 
such that 
$\supp\chi(\cdot,t)$ is compact for any $t\in[0,\tau]$,  
$H[\chi]\in C(\R^N\times[0,\tau]
\times(\R^{N}\setminus\{0\})\times\cS^{N})$.

\item[(H3)]
The functions 
$H[\chi]$ satisfy 
(A6) with $H=H[\chi]$ 
uniformly for any $\chi\in\Li(\cQ,[0,1])$.

\item[(H4)]
The functions 
$H[\chi]$ satisfy 
(A7) with $H=H[\chi]$ 
uniformly for any $\chi\in\Li(\cQ,[0,1])$. 

\item[(H5-(i))]
For any $R>0$, there exists $C_{H}>0$ such that 
\begin{align}
{}&
|H[\chi_{1}](x,t,p,X) 
-H[\chi_{2}](y,t,p,Y)| \nonumber\\
\le \, & 
C_{H}\bigl(\frac{|x-y|^{4}}{\ep^{4}}+ 
\kap_{\chi_1, \chi_2}(x,t)
+\frac{\kap_{\chi_1, \chi_2}^{2}(x,t)|x-y|^{2}}{\ep^{4}}
+\rho\|A^{2}\| 
\bigr) 
 \label{assump main ineq}
\end{align}
for any $\chi_{1},\chi_{2}\in\Li(\cQ,[0,1])$ 
and 
for any $\rho,\ep\in(0,1)$, 
$x,y\in \ol{B}(0,R)$, 
$t\in[0,T]$, 
$p=\ep^{-4}|x-y|^{2}(x-y)$ and 
$X,Y\in\cS^{N}$ satisfying 
\eqref{matrix ineq} and $A\in\cS^{N}$ given by \eqref{matrix-A}, 
where 
\[
\kap_{\chi_1, \chi_2}(x,t)
:=\int_{\R^N}|\chi_{1}(y,t)-\chi_{2}(y,t)|\,dy. 
\]

\item[(H5-(ii))]
One has inequality \eqref{assump main ineq} 
by replacing $\kap_{\chi_1,\chi_2}$ by 
\[
\ol{\kap}_{\chi_1, \chi_2}(x,t)
:=\int_0^{t}\int_{\R^N}G(x-y,t-s)|\chi_{1}(y,s)-\chi_{2}(y,s)|\,dyds, 
\]
where $G(x,t)$ is the Green function defined by 
\[
G(x,t):=\frac{1}{(4\pi t)^{N/2}}e^{-\frac{|x|^2}{4t}}. 
\]
\item[(H6)] For any $\chi\in\Li(\cQ,[0,1])$, if $\chi_{n}(x,t):=n\int_t^{t+1/n}\chi(x,s)ds$, then the nonlinearity
$$
H_{n}(x,t,p,X):=H[\chi_{n}](x,t,p,X)
$$
satisfies (A1)-(A4) and $u[\chi_{n}] \to u[\chi]$ uniformly in $\R^N \times [0,T]$ as $n \to + \infty$.
\end{itemize}

Assumptions (H1)--(H4) are modifications
of (A1)--(A7) in order to be able to deal with
the nonlocal equation \eqref{e-1}. 
While (H5-(i)) and (H5-(ii)) are specially designed to 
encompass dislocation type equations or FitzHugh-Nagumo 
type systems.  
Finally (H6) is the assumption which allows to use 
Theorem~\ref{continuous dependence} through an approximation argument 
(cf. Remark~\ref{appr-arg}). 
Further detailed examples are given in Section~\ref{sec:applic}.

We use the following definition of weak solutions introduced 
in \cite{BCLM1} which is inspired by \cite{I, N, Bo1,Bo2}. 
\begin{defn}[Definition of Weak Solutions]\label{def weak sol}
Let $u:\cQ\to\R$ be a continuous function. 
We say that $u$ is a weak solution of \eqref{e-1} 
if there exists 
$\chi\in\Li(\cQ,[0,1])$ such that 
\begin{itemize}
\item[{\rm (1)}]
$u$ is an $L^{1}$-viscosity solution 
of 
\begin{equation}\label{e-3}
\left\{
\begin{aligned}
&
u_t+H[\chi](x,t,Du,D^2u)=0 
& \textrm{in} \ \Q, \\
&u(\cdot,0)=u_0
& \textrm{in} \ \R^N, 
\end{aligned}
\right.
\end{equation}

\item[{\rm (2)}]
for almost every $t\in(0,T)$, 
\[
\mathbf{1}_{\{u(\cdot,t)>0\}}(x)\le 
\chi(x,t)\le 
\mathbf{1}_{\{u(\cdot,t)\ge0\}}(x) \ 
\textrm{for } a.e.  \ x\in\R^N. 
\]
\end{itemize}
Moreover, we say that $u$ is a classical solution of 
\eqref{e-1} if in addition, 
for almost all $t\in[0,T]$, 
\[
\mathbf{1}_{\{u(\cdot,t)>0\}}(x)
=
\mathbf{1}_{\{u(\cdot,t)\ge0\}}(x) \ \ \textrm{for a.e.} 
\ x\in\R^N. 
\]
\end{defn}

\begin{prop}[Weak Solutions are Classical in a Short Time]\label{weak is classical}
If there exists a weak solution 
$u\in C(\cQ)$ of \eqref{e-1}, 
then $u$ is classical in $\R^{N}\times(0,\ol{t}\land T)$ 
for some $\ol{t}>0$ which depends on 
$C_H,$ $L_H,$ $ \del_0,$ 
$\eta_0,$ $R_{T},$ $\|Du_{0}\|_{\infty},$ 
$\|\nu\|_{\infty},$ $\|D\nu\|_{\infty}$.  
\end{prop}

\begin{proof}
Let $(\chi,u)\in\Li(\cQ,[0,1])\times C(\cQ)$ 
be an $L^{1}$-viscosity solution of \eqref{e-3}. 
We prove that 
\begin{equation}\label{classic eq}
\mathbf{1}_{\{u(\cdot,t)>0\}}(x)
=
\mathbf{1}_{\{u(\cdot,t)\ge0\}}(x) 
\ \textrm{for a.e.} \ (x,t)\in\R^{N}\times(0,\ol{t}\land T) 
\end{equation}
for some $\ol{t}>0$.

%
%

We use (H6) and set $u_{n}:= u[\chi_n]$. We recall that $u_n$ is the viscosity solutions of 
\eqref{e-2} with $H=H_{n}$ for all $n\in\N$. 

By the comparison theorem for local equations, 
Proposition~\ref{regularity} 
we have 
\begin{gather*}
|u_{n}(x,t)|\le C \ \textrm{on} \ \cQ, \\ 
|u_{n}(x,t)-u_{n}(y,t)|\le C|x-y|, \quad 
|u_{n}(x,t)-u_{n}(x,s)|\le C|t-s|^{1/2}
\end{gather*}
for all $x,y\in\R^N$, $t,s\in[0,T]$ and some $C>0,$ 
which is independent of $n$. 
In view of Ascoli-Arzel\'{a} theorem, 
the stability (see \cite{Bo1,Bo2}) 
and the uniqueness (see \cite{N, Bo1,Bo2}) of 
$L^{1}$-viscosity solutions of \eqref{e-3}, 
we have 
$u_n\to u$ locally uniformly on $\cQ$ 
for $u\in C(\cQ)$ which is the $L^{1}$-viscosity solution 
of \eqref{e-3}.

Moreover, in view of Corollary~\ref{lower bound}, 
we have 
\[
-|Du_{n}(x,t)|\le -\frac{\eta(t)}{\|\nu\|_{\infty}} 
\quad\textrm{in} \ 
\{|u_{n}(\cdot,t)|<\frac{\del_{0}}{4}\}\times(0,\ol{t}\land T) 
\]
for some $\ol{t}>0$. 
By the usual stability result of viscosity solution, 
we get 
\[
-|Du(x,t)|\le -\frac{\eta(t)}{\|\nu\|_{\infty}} 
\quad\textrm{in} \ 
\{|u(\cdot,t)|<\frac{\del_{0}}{4}\}\times(0,\ol{t}\land T), 
\]
which implies that 
$\cL^{N}(\{u(\cdot,t)=0\})=0$ 
for a.e. $t\in(0,\ol{t}\land T)$ 
in view of \cite[Corollary 1 in p. 84]{EG}. 
Therefore, we get \eqref{classic eq}. 
\end{proof}

\begin{rem}
By Proposition \ref{weak is classical}, 
we have 
$\mathbf{1}_{\{u(\cdot,t)\ge0\}}\in 
C([0,\ol{t}\land T];L^{1}(\R^N))$ 
for any weak solution $u$ of \eqref{e-1}. 
In view of {\rm (H2)}, we see that 
$t\mapsto H[\mathbf{1}_{\{u\ge0\}}](x,t,p,X)$ 
is continuous on $[0,\ol{t}\land T]$ 
for any $(x,p,X)\in\R^N\times(\R^N\setminus\{0\})
\times\cS^N$. 
\end{rem}

We state our main result. 
\begin{thm}[Uniqueness Result of Solutions in a Short Time]\label{uniqueness}
If there exist weak solutions of 
the initial-value problem \eqref{e-1}, 
they are classical and unique in $\R^N\times[0,\ol{t}],$ 
where $\ol{t}$
is given by Theorem~{\rm \ref{key estimate}}.  
 \end{thm}

We formulate the main ingredient 
of the proof of the above theorem 
as a lemma. 

\begin{lem}\label{lem-for-unique}
Let $\ul{t}>0$ and $\eta$ be a continuous function 
on $[0,\ul{t}]$ such that $\eta(t)\ge\ul{\eta}>0$ 
for any $t\in[0,\ul{t}]$ and 
$u:\R^{N}\times[0,\ul{t}]\to\R$ be a bounded Lipschitz continuous function 
with respect to $x$ variable which satisfies 
\eqref{hyplip}, \eqref{umoins1} and \eqref{main-estimate}
on $[0,\ul{t}].$ 
Then we have 
\begin{gather}
\int_{\R^{N}}
\mathbf{1}_{\{-\del\le u(\cdot,t)<0\}}(y)\,dy 
\le \frac{M_4\del}{\ul{\eta}}, \label{dislo-ineq}\\
\int_{0}^{t}\int_{\R^N}
G(x-y,t-s)
\mathbf{1}_{\{-\del\le u(\cdot,s)<0\}}(y)\,dyds 
\le \frac{M_5\del}{\ul{\eta}} \label{FN-ineq}
\end{gather}
for any 
\begin{gather*}
\del\in(0, 
\min\{ \frac{\del_{0}}{4}, 
\frac{\ul{\eta}\delta_0}{4 L\|\nu\|_\infty},
\ul{\eta}\ol{\lam}
\}], 
\end{gather*}
where $M_4$ is a constant depending on 
$N, R, L, \|D\nu\|_{\infty}$ and 
$M_5$ is a constant depending on 
$N, R$, $\lam_{0}$, $\del_{0}$, 
$L$, $\|\nu\|_\infty$, $\|D\nu\|_\infty$. 
\end{lem}

\begin{proof}
We first prove the estimate \eqref{dislo-ineq}. 
We have 
\begin{eqnarray}\label{1-estim}
\int_{\R^{N}}
\mathbf{1}_{\{-\del\le u(\cdot,t)<0\}}(y)\,dy 
&= &
{\mathcal{L}}^N (\{-\del\le u(\cdot,t)\})
- {\mathcal{L}}^N(\Omega_t) 
\end{eqnarray}
for $\del>0$, since 
$\Omega_t:=\{ u(\cdot,t)>0\}\subset 
\{-\del\le u(\cdot,t)\}.$

We claim that
\begin{eqnarray}\label{claim1234}
\{-\del\le u(\cdot,t)\}\subset
(I+\frac{\del}{\ul{\eta}}\nu)^{-1}(\Omega_t)
\end{eqnarray}
for $t\in [0,\ul{t}]$  
and $\del$ small enough. 
We recall that $\psi_\lambda=(I+\lam\nu)$ is a $C^1$-diffeomorphism
when $\lam$ satisfies \eqref{deflambda0}. 
To prove the claim, let $(x,t)\in \R^N\times [0,\ul{t}]$ 
such that $u(x,t)\geq -\del$ 
and set 
$$
\lam :=\frac{\delta}{\ul{\eta}}.
$$
We distinguish two cases.
If $u(x,t)\geq \delta_0/4,$ then, by \eqref{hyplip}, 
\begin{eqnarray*}
u(x+\lam \nu(x),t)\geq u(x,t) -\lam L\|\nu\|_\infty
\ge \frac{\del_{0}}{4}-\lam L\|\nu\|_\infty
\geq 0
\end{eqnarray*}
for $\lam \leq \delta_0/(4L\|\nu\|_\infty).$
If $-\del\le u(x,t)\leq \delta_0/4,$ then,
by \eqref{main-estimate},
\begin{eqnarray*}
u(x+\lam \nu(x),t)\geq u(x,t)+\lam\eta(t)
\ge u(x,t)+\lam \ul{\eta} 
\geq -\del+\lam \ul{\eta}=0
\end{eqnarray*}
for $\del\leq \delta_0/4,$
$\lam \leq \ol{\lam}$ and $t\in[0,\ul{t}].$ 
Finally, \eqref{claim1234} holds if $\del$ is such that
$$
\del\le \min\{\frac{\del_{0}}{4}, 
\frac{\ul{\eta}\delta_0}{4 L\|\nu\|_\infty},
\ul{\eta}\ol{\lam}\}.
$$
By a change of variable, 
we have
\begin{eqnarray*}
{\mathcal{L}}^N ((I+\lam\nu)^{-1}(\Omega_t))
=\int_{\Omega_t} {\rm det}(D(I+\lambda\nu)^{-1})dx
\leq (1+2N\lam \|D\nu\|_\infty){\mathcal{L}}^N (\Omega_t^i), 
\end{eqnarray*}
for small $\delta$ and therefore small $\lam$,
since 
\[
{\rm det}(D(I+\lambda\nu)^{-1}(x))
= \left({\rm det}(I+\lambda D\nu(x))\right)^{-1}
=1-\lam\, {\rm tr}(D\nu)+o(\lam).
\]
From \eqref{1-estim} and \eqref{claim1234}, it follows
\begin{align*}
\int_{\R^{N}}
\mathbf{1}_{\{-\del\le u(\cdot,t)<0\}}(y)\,dy 
&\le 
{\mathcal{L}}^N ((I+\frac{\delta}{\ul{\eta}}\nu)^{-1}(\Omega_t))
- {\mathcal{L}}^N(\Omega_t)\\
{}&\leq 
2N\lam\|D\nu\|_{\infty} {\mathcal{L}}^N(\Om_{t})\\
{}&\leq 
\frac{2N\|D\nu\|_{\infty}{\mathcal{L}}^N(B(0,R))}{\ul{\eta}}\del
=:\frac{M_4\del}{\ul{\eta}} 
\end{align*}
by \eqref{umoins1}.

We next prove the estimate \eqref{FN-ineq}. 
Note that \eqref{main-estimate} implies 
the lower gradient estimate 
\[
-|Du(x,t)|\le -\frac{\ul{\eta}}{\|\nu\|_{\infty}} 
\quad\textrm{in} \ 
\{|u(\cdot,t)|<\frac{\del_{0}}{4}\}\times(0,\ul{t}). 
\]
From the increase principle
of \cite[Lemma 2.3]{BL}, we get
$$
\{-\del\le u(\cdot,t)<0\} 
\subset
(\ol{\Om}_{t}+\frac{\|\nu\|_\infty\del}{\ul{\eta}}B)
\setminus \ol{\Om}_{t},
$$ 
where 
$B:=B(0,1)\subset\R^{N}$. 
Therefore, 
noting that \eqref{main-estimate} implies that 
$\Om_{t}$ has a interior cone property as we can see 
in the proof of Corollary \ref{cone}, 
by \cite[Lemma 4.4]{BCLM2}
for some 
$M_{5}$ depending on
$N, R$, $\lam_{0}$, $\del_{0}$, $\ul{\eta}$, 
$L$, $\|\nu\|_\infty$, $\|D\nu\|_\infty$, 
we have 
\begin{align*}
{}&
\int_{0}^{t}\int_{\R^N}
G(x-y,t-s)
\mathbf{1}_{\{-\del\le u(\cdot,s)<0\}}(y)\,dyds\\
\le & \ 
\int_{0}^{t}\int_{\R^N}
G(x-y,t-s)
\mathbf{1}_{(\ol{\Om}_{t}+\|\nu\|_\infty 
\del B/\ul{\eta})\setminus 
\ol{\Om}_{t}}(y)\,dyds\\
= & \ 
|\phi(x,t,\frac{\|\nu\|_\infty\del}{\ul{\eta}})-\phi(x,t,0)|
\le \frac{M_{5}\del}{\ul{\eta}}, 
\end{align*}
where 
\[
\phi(x,t,r):=
\int_{0}^{t}\int_{\R^N}
G(x-y,t-s)
\mathbf{1}_{\ol{\Om}_{s}+rB}(y)\,dyds. 
\]
\end{proof}

\begin{proof}[Proof of Theorem {\rm \ref{uniqueness}}]
Suppose that 
there exist viscosity solutions $u_{1}$ and $u_{2}$ of 
\eqref{e-1}. 
Let $\tau\in(0,T]$
which will be fixed later and
set
\[
\del_{\tau}:=\max_{\R^{N}\times[0,\tau]}|(u_{1}-u_{2})(x,t)|. 
\]

In view of Theorem \ref{continuous dependence} and 
(H5-(i)) or (H5-(ii)), 
we have 
\begin{eqnarray}\label{kaptau}
\del_{\tau}
\le 
M_{1}\kap_\tau(\tau+\tau^{1/2})
 \quad \textrm{or} \quad  
\del_{\tau}
\le 
M_{1}\ol{\kap}_\tau(\tau+\tau^{1/2}), 
\end{eqnarray}
where 
\begin{gather}
\kap_\tau:=
\sup_{t\in[0,\tau]}
\int_{\R^N}
|\mathbf{1}_{\{u_{1}(\cdot,t)\ge0\}}(y)-
\mathbf{1}_{\{u_{2}(\cdot,t)\ge0\}}(y)|\,dy 
\label{dislocation}\\
\textrm{and} \nonumber\\
\ol{\kap}_\tau:=
\sup_{x\in\R^{N}, t\in[0,\tau]}
\int_{0}^{t}\int_{\R^N}
G(x-y,t-s)|\mathbf{1}_{\{u_{1}(\cdot,s)\ge0\}}(y)
-\mathbf{1}_{\{u_{2}(\cdot,s)\ge0\}}(y)|\,dyds.
\label{FN}
\end{gather}
Note that 
\[
|\mathbf{1}_{\{u_1(\cdot,t)\ge0\}}(y)
-\mathbf{1}_{\{u_2(\cdot,t)\ge0\}}(y)|
\le 
\mathbf{1}_{\{-\del_{\tau}\le u_1(\cdot,t)<0\}}(y)
+\mathbf{1}_{\{-\del_{\tau}\le u_2(\cdot,t)<0\}}(y).  
\]
We fix
$$
t^\ast\in(0, \ol{t}\land T),
$$
where $\ol{t}$ is given by Theorem~\ref{key estimate}.
Take $\tau\leq t^\ast$ small enough
in order that $\del_\tau\leq \del_0/4,$ 
the lower-bound gradient estimate
(Corollary~\ref{lower bound}) holds on $[0,\tau]$ and,
for all $t\in [0,\tau],$ $\eta(t)\geq \eta(t^\ast)=:\ol{\eta}>0.$ 
Moreover, take $\tau\le t^{\ast}$ such that 
\[
\del_{\tau}\in(0, 
\min\{ \frac{\del_{0}}{4}, 
\frac{\ol{\eta}\delta_0 e^{-KT}}{4 \|\nu\|_\infty\|Du_{0}\|},
\ol{\eta}\ol{\lam}
\}], 
\]
where $K$ is the constant give by Proposition \ref{regularity}. 
By continuity of $u_1$, $u_2$
which achieve the same initial condition $u_0$, 
it is always possible to find $\tau>0$ small enough
in order that the above condition holds.

By Lemma \ref{lem-for-unique}, we have 
\[
\kap_{\tau}\le 
\frac{C_1\del_{\tau}}{\ol{\eta}}
\]
for some $C_1$ depending on 
$N, R_{T}, C_H, \|Du_{0}\|_{\infty}, 
\|D\nu\|_{\infty}$ and 
\[
\ol{\kap}_{\tau}\le 
\frac{C_2\del_{\tau}}{\ol{\eta}}
\]
for some $C_2$ depending on 
$C_{H}, \lam_{0},\del_{0},\eta_{0},\|Du_{0}\|_{\infty}
\|\nu\|_\infty, \|D\nu\|_\infty$.

Therefore, we get 
\begin{equation*}
\del_{\tau}
\le 
\frac{C}{\ol{\eta}}\del_{\tau}(\tau+\sqrt{\tau})
\end{equation*}
for some constant $C(C_H, \lam_{0}, \del_0, \eta_0, R_{T}, 
\|Du_{0}\|_{\infty},N,
\|\nu\|_\infty, \|D\nu\|_\infty)>0$ which is independent of 
$\tau$. 
For $\tau$ small enough, we have $\del_{\tau}=0$. 
It follows $u_{1}=u_{2}$ on 
$\R^{N}\times[0,\tau]$. 

We consider 
$\ol{\tau}=\sup\{\tau>0\mid u_{1}=u_{2} \ 
\textrm{on} \ \R^{N}\times[0,\tau]\}$. 
If $\ol{\tau}<t^{\ast}$, 
then we can repeat the above proof from time 
$\ol{\tau}$ instead of $0$. 
Finally, we have 
$u_{1}=u_{2}$ on 
$\R^{N}\times[0,t^{\ast}]$ for all $t^{\ast}<\ol{t},$
which gives the conclusion. 
\end{proof}


\section{Applications}
\label{sec:applic}

In the companion paper \cite{BCLM3}, 
the framework to show existence of weak solutions 
of \eqref{e-1} is given and 
as applications, 
existence results for weak solutions of level set equations 
appearing in dislocations' theory and in the study of 
FitzHugh-Nagumo systems are presented 
(see \cite[Sections 3.2, 4.2]{BCLM3}). 
In this section, we give uniqueness results 
for viscosity solutions of such equations.

\subsection{Dislocation Type Equations}
We consider 
the level set equation of 
the evolution of hypersurfaces:  
\begin{equation}\label{dislocation evo}
V=
M(n(x))\bigl(c_{0}(\cdot,t)\ast\mathbf{1}_{\ol{\Om}_{t}}(x)+c_{1}(x,t)
-\Div_{\Gam_{t}}(\xi(n)(x))\bigr) 
\quad\textrm{on} \ 
\Gam_{t}, 
\end{equation}
where we use the notations in Introduction. 
Here $M:\bS^{N-1}\to\R^{+}$, 
$c_0, c_1:\cQ\to\R$, 
$\xi:\R^{N}\to\R^{N}$ 
are given functions which satisfy 
the following assumption {\rm (A)}: 
\begin{itemize}
\item[(i)]
$M\in C(\bS^{N-1})$, 
$0\le M(\hat{p})\le\ol{M}$  
for all $p\in\R^{N}\setminus\{0\}$, 
where $\hat{p}=p/|p|$ and 
$\hat{p}\mapsto\sqrt{M(\hat{p})}$ 
is Lipschitz continuous; 
\item[(ii)]
$c_{0}\in C([0,T]; L^{1}(\R^n))$, $c_1\in C(\cQ)$, 
$D_{x}c_{0}\in\Li([0,T];L^{1}(\R^N))$; 
\item[(iii)]
there exist constants $L_{c}, M_{c}>0$ 
such that, for any $x,y\in\R^N$ and $t\in[0,T]$
\begin{gather*}
\|D_{x}c_{0}(\cdot,t)\|_{L^{1}(\R^{N})}\le L_{c}, \quad 
|c_1(x,t)-c_1(y,t)|\le L_{c}|x-y|,\\ 
\|c_{0}(\cdot,t)\|_{L^{1}(\R^N)}+|c_{0}(x,t)|
+|c_{1}(x,t)|\le M_{c}; 
\end{gather*}
\item[(iv)]
$\xi(p)=(\xi^{1}(p),\ldots,\xi^{N}(p))
=D\gam(p)=(\pl\gam(p)/\pl p_{1},\ldots,\pl\gam(p)/\pl p_{N})$ 
for some positively homogeneous function 
$\gam\in C^{2}(\R^N\setminus\{0\})$ 
of degree $1$, 
i.e., $\gam(\al p)=\al\gam(p)$ 
for $\al>0$, $p\in\R^N\setminus\{0\}$, 
which satisfies 
\[
D^{2}\gam(p)=\left(
\begin{array}{ccc}
\frac{\pl^{2}\gam(p)}{\pl p_{1}\pl p_{1}} & 
\cdots & 
\frac{\pl^{2}\gam(p)}{\pl p_{1}\pl p_{N}} \\
\vdots & {} & \vdots \\
\frac{\pl^{2}\gam(p)}{\pl p_{N}\pl p_{1}} & 
\cdots & 
\frac{\pl^{2}\gam(p)}{\pl p_{N}\pl p_{N}}
\end{array}
\right)
\ge 0 
\ 
\textrm{for all} \ 
p\in\R^{N}\setminus\{0\}, 
\]
\[
\sup_{p\in\R^{N}\setminus\{0\}}|D^{2}\gam(\hat{p})|<\infty \ 
\textrm{and} \ 
\hat{p}\mapsto
\bigl(D^{2}\gam(\hat{p})^{1/2}\bigr) 
\ \textrm{is Lipschitz continuous}. 
\]
\end{itemize}

The function $\xi$ is called the \textit{Cahn-Hoffman} vector 
and the last term 
in the right hand side of \eqref{dislocation evo} 
is the anisotropic curvature of $\Gam_{t}$ at $x$ 
given by 
$\Div_{\Gam_{t}}
(\xi(n)(x))
=\tr\bigl((I-n(x)\otimes n(x))
D_{x}(\xi(n))(x) 
\bigr)$. 
We refer the reader to the monograph by Giga \cite{G} and the references
therein for more details.

For reader's convenience, 
we derive the level set equation of \eqref{dislocation model},
see \cite{AHBM, G}. 
We have 
\begin{equation}\label{aniso-1}
\Div_{\Gam_{t}}(\xi(n)(x))=
\tr(
D_{x}(\xi(n))(x)
)
-\tr(n(x)\otimes n(x) 
D_{x}(\xi(n))(x)
). 
\end{equation}
Since $\gam$ is positively homogeneous of degree $1$, 
$\xi^{i}$ is positively homogeneous of degree $0$ 
for all $i\in\{1,\ldots,N\}$, i.e.,  
$\xi^{i}(\al p)=\xi^{i}(p)$ 
for all $\al>0$, $p\in\R^{N}\setminus\{0\}$. 
Differentiating in $\al$, setting $\al=1$
and noting that $D_{p}\xi(p)=D^{2}\gam(p)\in\cS^{N},$
yields 
\begin{equation}\label{aniso-2}
\sum_{j=1}^{N}\xi^{i}_{p_j}(p)p_{j}
=\sum_{j=1}^{N}\xi^{j}_{p_i}(p)p_{j}
=0 
\quad\textrm{for all} \ 
i\in\{1,\ldots,N\}, 
\end{equation}
where 
$\xi^{i}_{p_j}(p)=\pl\xi^{i}(p)/\pl p_{j}$. 
Equality \eqref{aniso-2} yields 
\begin{equation*}
\tr(n(x)\otimes n(x) D_{x}(\xi(n))(x))
= 
\sum_{i,j,k} n^{k}(x)n^{j}(x)\xi^{j}_{p_i}(n(x))n^{i}_{x_k}(x)=0,
\end{equation*}
where $n^{i}_{x_{k}}(x)=\pl n^{i}(x)/\pl x_{k},$ and 
\[
R_{p}D^{2}\gam(p)=D^{2}\gam(p)=D^{2}\gam(p)R_{p}, 
\]
where $R_{p}=I-\hat{p}\otimes \hat{p}$. 
Introducing an auxiliary function $u:\cQ\to\R$ 
such that 
$u(\cdot,t)=0$ on $\Gam_{t}$ 
and 
$u(\cdot,t)>0$ in $\Om_{t}$ 
for all $t\in[0,T]$, 
we note $n(x)=-Du(x)/|Du(x)|$ and it follows 
\begin{align*}
\Div_{\Gam_{t}}(\xi(n)(x))
&=\, 
\Div_{x}\xi(-Du(x))
=-\tr\bigl(D^{2}\gam(-Du(x))D^{2}u(x)\bigr)\\
&=\, 
-\tr\bigl(R_{Du(x)}
D^{2}\gam(-Du(x))
R_{Du(x)}D^{2}u(x)\bigr)\\
&=\, 
\frac{-1}{|Du(x)|}
\tr\bigl(R_{Du(x)}
D^{2}\gam\bigl(-\frac{Du(x)}{|Du(x)|}\bigr)
R_{Du(x)}
D^{2}u(x)\bigr). 
\end{align*}
In the above equalities, we used the homogeneity 
of degree $0$ of $\xi$ and 
of degree $-1$ of $D^{2}\gam$. 
Set 
\begin{gather*}
c[\chi](x,t,p):= 
M(-\hat{p})\bigl(c_0(\cdot,t)\ast\chi(\cdot,t)(x)
+c_1(x,t)\bigr), \\
\sig(p):= 
\sqrt{M(-\hat{p})}
\bigl(D^{2}\gam(-\hat{p})\bigr)^{1/2}
R_{p}, \\
H_{d}[\chi](x,t,p,X):= 
-c[\chi](x,t,p)|p|
-\tr\bigl(\sig(p)\sig^{T}(p)X\bigr), 
\end{gather*}
for any $\chi\in\Li(\cQ,[0,1])$, 
$(x,t,p,X)\in\cQ\times(\R^N\setminus\{0\})\times\cS^{N}$. 
The level set equation 
of \eqref{dislocation evo} is the equation 
\[
u_t+H_{d}[\mathbf{1}_{\{u(\cdot,t)\ge0\}}]
(x,t,Du,D^{2}u)=0 
\quad\textrm{in} \ 
\Q, 
\]
which is a particular case of \eqref{e-1}. 

\begin{thm}\label{dislocation thm}
Under assumptions {\rm (A)}, {\rm (I1)} 
and {\rm (I2)}, 
the initial value problem 
\eqref{e-1} with $H=H_{d}$ 
has at least a weak solution in $\cQ$. 
Moreover, 
weak solutions are classical and unique 
in $\R^{N}\times[0,t^{\ast}]$ 
for some $t^{\ast}\in(0,T]$ 
which depends only on 
$L_{c}, M_{c}, \ol{M}, \del_0, \eta_0, R_0,$  $\|Du_{0}\|_{\infty}$
and $\|\nu\|_{\infty}$.  
\end{thm}

We refer to \cite{F, FM} for the short time existence and 
uniqueness of the solution of a dislocation dynamics equation 
with a mean curvature term under different assumptions.

\begin{proof}
The existence of weak solutions is proved 
in \cite[Theorem 3.3]{BCLM3}. 
By using Theorem \ref{uniqueness}, 
we prove a short time uniqueness. 
It is easy to check that 
(H2), (H3) are satisfied. 
Due to the arguments in 
Example \ref{ex-2} and assumptions (A) (i), (iv), 
we see that (H4) is satisfied. 
We prove that 
$H_{d}$ satisfies (H1) and (H5-(i)). 
We first check (H5-(i)). 
Set 
$\tilde{c}_{i}(x,t,p):=c[\chi_{i}](x,t,p)$ 
for $(x,t,p)\in\cQ\times(\R^N\setminus\{0\})$, 
$\chi_{i}\in\Li(\cQ,[0,1])$ and 
$i=1,2$. 
Note that 
\begin{align*}
{}&\, 
|\tilde{c}_{1}(x,t,p)-\tilde{c}_{2}(x,t,p)|\\
\le&\, 
M(-\hat{p})
\int_{\R^N}|c_0(x-y)(\chi_{1}(y,t)-\chi_{2}(y,t))|\,dy\\
\le&\, 
\ol{M}M_{c}
\int_{\R^N}|\chi_{1}(y,t)-\chi_{2}(y,t)|\,dy. 
\end{align*}

We finally check (H1). 
Let $(u,\chi)\in C(\cQ)\times\Li(\cQ,[0,1])$ 
be a $L^{1}$-viscosity solution of \eqref{e-3}.
We extend the functions 
$c[\chi](x,\cdot,p)$ to be 
equal $0$ when $t<0$ and $T\le t$. 
We set 
\begin{gather*}
c^{n}(x,t,p):=
c[\chi](x,\cdot,p)\ast\zeta_{n}(t), 
\end{gather*}
for all $(x,t,p)\in\cQ\times(\R^N\setminus\{0\})$, 
where 
$\zeta_{n}(r):=n\zeta(nr)$ for $r\in\R$ 
and $\zeta$ is a standard mollification kernel. 
Then we have 
$c^{n}\in C(\cQ\times(\R^N\setminus\{0\}))$, 
\begin{gather*}
|\int_{0}^{t}c^{n}(x,s,p)-c[\chi](x,s,p)\,ds|
\to 0, 
\end{gather*}
locally uniformly in $(0,T)$ as $n\to\infty$. 
Moreover, the $c^{n}$'s 
are Lipschitz continuous with respect to $x$ variable 
with the constant $L_{c}$ and 
bounded (independently of $n$). 
Let $u_{n}$ be the viscosity solutions of 
\begin{equation*}
\left\{
\begin{aligned}
&
u_t-c^{n}(x,t,Du)|Du|
-\tr(\sig(Du)\sig^{T}(Du)D^{2}u)
=0 
& \textrm{in} \ \Q, \\
&u(\cdot,0)=u_0
& \textrm{in} \ \Q, 
\end{aligned}
\right.
\end{equation*}
for all $n\in\N$. 
By Example \ref{ex-2} and 
Proposition  \ref{regularity}, 
we have 
\begin{equation}\label{propun}
\begin{aligned}
& |u_{n}(x,t)|\le 1 \ \textrm{on} \ \cQ, \quad 
u_{n}(x,t)=-1 \ \textrm{on} \ (\R^N\setminus B(0,R(t)))\times[0,T], \\
& |u_{n}(x,t)-u_{n}(y,t)|\le C|x-y|, \quad 
|u_{n}(x,t)-u_{n}(x,s)|\le C|t-s|^{1/2},
\end{aligned}
\end{equation}
for all $x,y\in\R^N$, $t,s\in[0,T]$ 
and some $C>0$, 
where $R(t)$ is the function introduced in Example~\ref{ex-2}. 
By Ascoli-Arzel\'{a} theorem, the uniqueness of 
$L^{1}$-viscosity solutions of \eqref{e-3} 
with $H[\chi]=H_{d}[\chi]$
and \cite[Theorem 1.1]{B}, 
we have 
$u_n\to u$ locally uniformly on $\cQ$ and 
$u$ still satisfies properties~\eqref{propun}.

The above mollification argument can be an alternative way of getting the approximation property we need (cf. (H6)) by a regularization of $H[\chi]$ by $H[\chi](x,\cdot,p,X)\ast\zeta_{n}(t)$ instead of $H[\chi\ast\zeta_{n}(t)](x,t,p,X)$. 
\end{proof}

\subsection{A FitzHugh-Nagumo Type System}

We consider the following system: 
\begin{numcases}
{}
u_t=\Bigl(\al(v)+\Div\bigl(\frac{Du}{|Du|}\bigr)\Bigr)|Du| 
& in $\Q$, \nonumber \\
v_t-\Del v=
g^{+}(v)\mathbf{1}_{\{u\ge0\}}+
g^{-}(v)(1-\mathbf{1}_{\{u\ge0\}}) 
& in $\Q$,\label{FN asymp} \\
u(\cdot,0)=u_{0}, \ 
v(\cdot,0)=v_{0}
& in $\R^N$, \nonumber
\end{numcases}
which is obtained as the asymptotic as $\ep\to0$ of 
the following Fitzhugh-Nagumo system 
arising in neural wave propagation or chemical kinetics 
(see \cite[Theorem 4.1]{SS}): 
\begin{equation}\label{FN system}
\left\{
\begin{aligned}
u_{t}^{\ep}-\Del u^{\ep} 
& = \frac{1}{\ep^{2}}f(u^{\ep}, v^{\ep}) 
& \textrm{in} \ \Q, \\
v_{t}^{\ep}-\Del v^{\ep} 
& = g(u^{\ep}, v^{\ep}) 
& \textrm{in} \ \Q, 
\end{aligned}
\right.
\end{equation}
where 
\begin{equation*}
\left\{
\begin{aligned}
f(u,v)&=
u(1-u)(u-a)-v 
& {\rm (}0<a<1{\rm )}, \\
g(u,v)&=
u-\gam v 
& {\rm (}\gam>0{\rm )}. 
\end{aligned}
\right.
\end{equation*}
The functions 
$\al, g^{+}$ and $g^{-}:\R\to\R$ appearing in 
\eqref{FN asymp} are associated with $f$ and $g$. 

We make the following assumptions (B): 
\begin{itemize}
\item[(i)]
the function $v_{0}$ is bounded and of class $C^{1}$ with 
$\|Dv_{0}\|_{\infty}<+\infty$; 

\item[(ii)]
the functions 
$g^{-},g^{+}$ are Lipschitz continuous on $\R$ 
with a Lipschitz constant $L_{g}\ge0$ and 
there exist $\ul{g}, \ol{g}\in\R$ such that 
\[
\ul{g}\le g^{-}(r)\le g^{+}(r)\le\ol{g} \quad 
\textrm{for all} \ r\in\R; 
\]

\item[(iii)]
$|\al(r)-\al(s)|\le L_{\al}|r-s|$, 
$|\al(r)|\le M_{\al}$ 
for all $r,s\in\R$ and 
some $L_{\al}, M_{\al}>0$. 
\end{itemize}

For $\chi\in\Li(\cQ,[0,1])$, 
we write 
$v$ for the solution of 
\begin{equation}\label{diffusion part}
\left\{
\begin{aligned}
&
v_{t}-\Del v 
= g^{+}(v)\chi+
g^{-}(v)(1-\chi) 
&& \textrm{in} \ \Q, \\
&
v(\cdot,0)=v_{0}
&& \textrm{in} \ \R^{N}, 
\end{aligned}
\right.
\end{equation}
and set 
$c[\chi](x,t):=\al (v(x,t))$.  
Then Problem \eqref{FN asymp} reduces to 
\begin{equation}\label{FN nonlocal}
\left\{
\begin{aligned}
&
u_{t}-\Bigl(c[\mathbf{1}_{\{u\ge0\}}](x,t)  
+\Div\bigl(\frac{Du}{|Du|}\bigr)\Bigl)|Du(x,t)|
=0
&& \textrm{in} \ \Q,  \\
&u(\cdot,0)=u_0
&& \textrm{in} \ \R^{N}, 
\end{aligned}
\right.
\end{equation}
which is a particular case of \eqref{e-1}. 

\begin{thm}\label{FN thm}
Under assumptions {\rm (B)}, 
{\rm (I1)} and {\rm (I2)}, 
the initial value problem 
\eqref{FN nonlocal} 
has at least a weak solution in $\cQ$. 
Moreover, it is classical and unique 
in $\R^{N}\times[0,t^{\ast}]$ 
for some $t^{\ast}$ 
which depends only on 
$L_{\al}, M_{\al}, \del_0, \eta_0, R_0,$  $\|Du_{0}\|_{\infty}$
and $\|\nu\|_{\infty}$.  
\end{thm}

\begin{proof}
The existence of weak solutions is proved 
in \cite[Theorem 3.4]{BCLM3}. 
See also \cite{GGI,SS}. 
It is easy to check that 
(H2) and (H3) are satisfied. 
Due to similar arguments as those in the proof of 
Theorem \ref{dislocation thm}, 
we see that (H1) and (H4) are satisfied. 
We prove that 
(H5-(ii)) is satisfied. 
For $\chi_{1}, \chi_{2}\in \Li(\cQ,[0,1])$, 
the solutions of \eqref{diffusion part} are given by 
\begin{align*}
v_{i}(x,t)=\,&
\int_{\R^N}G(x-y,t)v_{0}(y)\,dy \, + \\
{}\, &\int_{0}^{t}\int_{\R^N}G(x-y,t-s)
(g^{+}(v_{i})\chi_{i}(y,s)
+g^{-}(v_{i})(1-\chi_{i}(y,s)))\,dyds 
\end{align*}
for $i=1,2$. 
By the proof of \cite[Theorem 4.1]{BCLM2}, we have 
\[
|c[\chi_{1}](x,t)-c[\chi_{2}](x,t)|\le 
L_{\al}(\ol{g}-\ul{g})e^{3L_{g}T}
\int_{0}^{t}\int_{\R^N}
G(x-y,t-s)
|\chi_{1}(y,s)-\chi_{2}(y,s)|\,dyds 
\]
for all $(x,t)\in\cQ$. 
This completes the proof. 
\end{proof}

\subsection{Nonlocal equations with volume-dependent terms}
We consider the following evolution of hypersurfaces: 
\begin{equation}\label{measure surface}
V=
\beta(\cL^{N}(\ol{\Om}_{t}))
-\Div_{\Gam_{t}}(n(x)) 
\quad\textrm{on} \ 
\Gam_{t}, 
\end{equation}
where the function $\beta:\R\to\R$ is 
Lipschitz continuous. 
A typical example is 
$\beta(r)=a+br$ for some $a,b\in\R$ 
which has been  studied by Chen, Hilhorst and Logak in 
\cite{CHL} (see \cite{C, CP} also). 
The authors prove that 
the limiting behaviour of 
the following reaction-diffusion equation 
\begin{numcases}
{}
u_t=\Del u+\frac{1}{\ep^{2}}f(u,\ep\int_{\Om}u) 
& in $\Om\times(0,\infty)$, \nonumber \\
\frac{\pl u}{\pl n}=0 
& on $\Om\times(0,\infty)$, \nonumber\\
u(\cdot,0)=g^{\ep} 
& in $\Om$ \nonumber
\end{numcases}
is characterized by the motion of hypersurface 
\eqref{measure surface}
with $\beta(r)=a+br$. 
The level set equation of \eqref{measure surface} 
is the nonlocal equation:
\begin{equation}\label{m-2}
\left\{
\begin{aligned}
&
u_t
-\Bigl(\beta\bigl(\cL^{N}(\{u(\cdot,t)\ge0\})\bigr)
+\gam\,\Div\bigl(\frac{Du}{|Du|}\bigr)\Bigl)|Du|=0
&&
\textrm{in} \ \Q, \\
&
u(\cdot,0)=u_0
&&
\textrm{in} \ \R^{N}, 
\end{aligned}
\right.
\end{equation}
for some positive constant $\gam.$  

It is worthwhile to mention that 
we cannot expect the global existence of weak solutions 
without any restriction of the growth of $\beta$, 
because the front can blow up at a finite time. 
A growth condition to ensure a global existence result
is given below in (C-(ii)), see
also \cite[Section 4.2]{BCLM3}.

We can easily check that the equation \eqref{m-2} 
satisfies assumptions (H1)-(H4) and (H5-(i)). 
So, the following result holds. 
\begin{thm}
Under assumptions 
{\rm (I1)} and {\rm (I2)}, 
for any $\gam, R>0$, there exists 
a constant $t_{R}\in(0,T]$ and 
at least a weak solution 
$(u,\chi)\in C(\cQ)\times\Li(\cQ,[0,1])$
of the initial value problem for 
\eqref{m-2} 
such that 
$\{x\in\R^N\mid 
u(x,t)\ge0\}\subset B(0,R)$ 
for any $t\in[0,t_{R}]$. 
Moreover, 
the weak solution is classical and unique 
in $\R^{N}\times[0,t^{\ast}]$   
for some $t^{\ast}\in(0,t_{R}]$ 
which depends only on 
the Lipschitz constant of $\beta$, 
$R$, $N$, $\del_0, \eta_0, R_0,$  $\|Du_{0}\|_{\infty}$
and $\|\nu\|_{\infty}$.  
\end{thm}

Now, adding the assumption (C): 
\begin{itemize}
\item[(i)]
the assumption (I2) holds with 
$\nu(x)=-x$ in $U_0$, 

\item[(ii)]
there exist $L_1, L_2>0$ such that 
\[
0\le\beta(r)\le L_{1}+L_{2}r^{1/N} 
\quad\textrm{for all} \ 
r\in[0,\infty), 
\]
\end{itemize}
we can show a global existence and uniqueness result 
of solutions of \eqref{m-2} for small $\gam$. 
More precisely, we get 
\begin{thm}\label{measure global thm}
Under assumptions
{\rm (I1)}, {\rm (I2)} and {\rm (C)}, 
there exists a positive constant 
$\ol{\gam}=\ol{\gam}(N,T,\eta_{0}/\|Du_{0}\|_{\infty})$ 
such that, for any $0\le \gam\le\ol{\gam}$, 
there exists a unique viscosity solution of 
\eqref{m-2} in $\cQ$. 
\end{thm}

\begin{lem}\label{measure key estimate}
Let $u$ be the viscosity solution of 
\begin{numcases}
{}
u_t=c(t)|Du|
+\gam\,\tr\big((I-\frac{Du\otimes Du}{|Du|^2})D^2u\big) 
& in $\Q$,\nonumber \\
u(\cdot,0)=u_0 & in $\R^N$, \nonumber
\end{numcases}
where 
$\gam$ is a positive constant and 
$c\in C([0,T])$ is a nonnegative given function. 
Then we have, 
for any $(x,t)\in\{|u(\cdot,t)|\le\del_{0}/4\}\times[0,T]$, 
\[
u((1-\lam) x,t)\ge 
u(x,t)+(\eta_{0}-C\|Du_{0}\|_{\infty}\sqrt{\gam t})\lam 
\quad 
\textrm{for all} \ 
\lam\in[0, \ol{\lam}], 
\]
where $C$ is a positive constant which depends only on $N$ and $H$
and $\ol{\lam}$ is the constant 
given by Lemma {\rm \ref{psi}}
{\rm (}We may assume that $\ol{\lam}\le1/2${\rm )}.
\end{lem}

The proof of Lemma \ref{measure key estimate} is very similar 
to that of Theorem \ref{key estimate}, but 
since we would like to explain 
how to use the positiveness of $c(t)$ and 
note the dependence of $C$ in Lemma~\ref{measure key estimate}, 
we give it here. 
 
\begin{proof}
Let $\Psi$ be the function given by Lemma \ref{psi} 
and 
fix $\lam\in[0,\ol{\lam}]$. 
Setting 
$v(x,t):=u((1-\lam) x,t)$ 
and 
$w(x,t):=\Psi(u(x,t)+\eta_{0}\lam)$ 
for all $(x,t)\in\cQ$, 
we consider 
\[
\sup_{x,y\in\R^N,t\in[0,T]}
\{w(x,t)-v(y,t)-
\frac{|x-y|^4}{\ep^4}-\tilde{K}t\}. 
\]
for $\ep,\tilde{K}>0$.  
By similar arguments as those used 
in Theorem \ref{continuous dependence}, 
there exist $(\ol{x},\ol{y},\ol{t})\in\ol{B}(0,R_{T}+1)^{2}\times[0,T]$ 
for some $R_{T}>0$ and small enough $\ep>0$ 
such that the supremum attains at $(\ol{x},\ol{y},\ol{t})$ 
and 
$(a,p,X)\in \ol{J}^{2,+}v(\ol{x},\ol{t})$ and 
$(b,p,Y)\in \ol{J}^{2,-}w(\ol{y},\ol{t})$ 
satisfy \eqref{ishii-lem}. 
Note that the Lipschitz constant of $u$ is $\|Du_{0}\|_{\infty}$ 
in this case. 
We have 
\[
|\ol{x}-\ol{y}|\le\|Du_{0}\|_{\infty}^{1/3}\ep^{4/3}, 
\quad
|p|\le 4\|Du_{0}\|_{\infty}. 
\]

We only consider the case where $\ol{x}\not=\ol{y}$. 
The definition of viscosity solutions immediately implies 
that we have 
\begin{eqnarray*}
&&
a\le c(\ol{t})|p|
+\gam\tr\bigl((I-\dfrac{p\otimes p}{|p|^2})X\bigr), \\
&&
b\ge 
\dfrac{c(\ol{t})}{1-\lam}|p|
+\dfrac{\gam}{(1-\lam)^2}\tr\bigl((I-\dfrac{p\otimes p}{|p|^2})Y\bigr).  
\end{eqnarray*}
It follows
\begin{eqnarray*}
a-b=\tilde{K}\le 
c(\ol{t})(1-\dfrac{1}{1-\lam})|p|+
\gam 
\tr\bigl((I-\dfrac{p\otimes p}{|p|^2})(X-\dfrac{Y}{(1-\lam)^2})\bigr).
\end{eqnarray*}
We note that
$I-p\otimes p/|p|^2$ is a positive definite bounded matrix and
\begin{eqnarray*}
X-\frac{Y}{(1-\lam)^2}\leq \left(\frac{-\lam}{1-\lam}\right)^2 
\dfrac{|\ol{x}-\ol{y}|^{2}}{\ep^{4}}I
+\rho\sum_{i=1}^{N}
\big\langle 
A^{2} 
\left(
\begin{array}{c}
e_{i} \\
e_{i} 
\end{array}
\right)
, 
\left(
\begin{array}{c}
e_{i} \\
e_{i} 
\end{array}
\right) 
\big\rangle 
\end{eqnarray*}
by \eqref{matrix ineq}.
In view of the positiveness of $c(t),$
we get,
for some constant $C=C(H,N)>0$ which may change line to line, 
\begin{align*}
\tilde{K}&\le 
\dfrac{C\gam\lam^2}{(1-\ol{\lam})^2}
\dfrac{|\ol{x}-\ol{y}|^{2}}{\ep^{4}}
+C\rho\|A^{2}\|\\
{}&\le 
\frac{C\|Du_{0}\|_{\infty}^{2/3}\gam\lam^2}{\ep^{4/3}}
+C\rho\|A^{2}\|.
\end{align*}
Sending $\rho\to0$ and taking 
$\tilde{K}
=(C+1)\|Du_{0}\|_{\infty}^{2/3}\gam\lam^2\ep^{-4/3}$, 
we necessarily have $\ol{t}=0$.

Thus we have for any $(x,t)\in\cQ$ 
\begin{align*}
{}&\, 
w(x,t)-v(x,t)\le 
w(\ol{x},0)-v(\ol{y},0)+\tilde{K}t
\le 
v(\ol{x},0)-v(\ol{y},0)+\tilde{K}t\\
&\le\, 
e^{4KT/3}\|Du_{0}\|_{\infty}^{4/3}\ep^{4/3}
+(C+1)\|Du_{0}\|_{\infty}^{2/3}\gam\lam^2 t \ep^{-4/3}. 
\end{align*}
Setting 
\[
\ep^{4/3}=\Bigl(\dfrac{(C+1)\gam\lam^2 t}{e^{4KT/3}\|Du_{0}\|_{\infty}^{2/3}}
\Bigr)^{1/2},
\]
we get 
\[
w(x,t)-v(x,t)\le 
\tilde{C}
\|Du_{0}\|\sqrt{\gam t}\lam, 
\]
where $\tilde{C}$ depends only on $N,H.$ 
This implies a conclusion. 
\end{proof}

\begin{proof}[Proof of Theorem {\rm \ref{measure global thm}}]
In \cite[Section 4]{BCLM3}, 
the global existence result for weak solutions of 
\eqref{m-2} is given. 
Due to Lemma \ref{measure key estimate}, 
we see that 
if 
\[
\gam\le
\frac{1}{T}\Bigl(\frac{\eta_{0}}{2C\|Du_{0}\|_{\infty}}\Bigr)^2, 
\]
then 
\[
\eta_{0}-C\|Du_{0}\|_{\infty}\sqrt{\gam t}\ge \frac{\eta_{0}}{2} 
\quad \textrm{for all} \ 
t\in[0,T]. 
\]
Therefore, we get 
\begin{equation}\label{lower gradient lem}
u((1-\lam)x,t)\ge 
u(x,t)+\frac{\eta_{0}}{2}\lam 
\quad \textrm{for all} \ 
(x,t)\in\{|u(\cdot,t)|\le\del_{0}/4\}\times[0,T], 
\ \lam\in[0,\ol{\lam}].
\end{equation}
A careful review of the proof of Theorems~\ref{uniqueness} 
gives the conclusion. 
\end{proof}


\begin{rem}\label{star perimeter}
Inequality \eqref{lower gradient lem} implies that 
the $r$-level set of $u,$ for $r$ close to 0, are star-shaped domains
with respect to a ball 
with center $0,$ see Lemma~\ref{I2geom}.
In particular, they are locally Lipschitz continuous graphs.  
Then we can get perimeter estimates without 
using Lemma \ref{perimeter}. 
Indeed, 
noting that, from the  lower gradient estimate (Corollary~\ref{lower bound})
and the increase principle~\cite[Lemma 2.3]{BL}, for small $\ep>0$, 
$
\{-\ep\le u(\cdot,t)\le0\}
\subset
\frac{\ol{\eta}}{\ol{\eta}-\ep}
\ol{\Om}_{t}
\setminus\Om_{t}, 
$
we have, 
for any $t\in[0,\ol{t}\land T]$,   
\begin{align*}
\frac{1}{\ep}\int_{-\ep}^{0}
\cH^{N-1}(\{u(\cdot,t)=r\})\,dr
&=\, 
\frac{1}{\ep}\int_{\{-\ep<u(\cdot,t)<0\}}|Du(x,t)|\,dx\\
{}&\le\, 
\frac{\|Du_{0}\|_{\infty}e^{Kt}}{\ep}
\cL^{N}(\{-\ep<u(\cdot,t)<0\})\\
{}&\le\, 
\frac{\|Du_{0}\|_{\infty}e^{Kt}}{\ep}
\cL^{N}(\frac{\ol{\eta}}{\ol{\eta}-\ep}
\Om_{t}\setminus\Om_{t})\\
{}&=\, 
\frac{\|Du_{0}\|_{\infty}e^{Kt}}{\ep}
\bigl(\bigl(
\frac{\ol{\eta}}{\ol{\eta}-\ep}
\bigr)^{N}-1\bigr)
\cL^{N}(\Om_{t}) 
\end{align*}
in view of the co-area formula 
and Proposition  \ref{regularity}.
Since the $\{u(\cdot,t)=r\}$'s are locally Lipschitz continuous,
the $(N-1)$-Hausdorff measure and the perimeter in Geometric
Measure Theory coincide (see \cite{EG}). Since this latter
perimeter is lower-semicontinuous with respect to the Hausdorff
convergence, sending $\ep\to0$, we get 
\[
\cH^{N-1}(\{u(\cdot,t)=0\})
\le 
\frac{N\|Du_{0}\|_{\infty}e^{Kt}}{\ol{\eta}}\cL^{N}(\Om_{t}) 
\]
for all $t\in[0,T]$. 
\end{rem}


\section{Appendix}

We give the proof of 
Lemma~\ref{I2geom},
Proposition~\ref{regularity} and
Lemma \ref{psi}.


\begin{proof}[Proof of Lemma {\rm \ref{I2geom}}]
Let $\Omega_0$ be defined by~\eqref{strict-etoile-boule}
and consider the function
\begin{equation*}
u_0(x)=\left\{
\begin{array}{ll}
{\rm dist}(x,\Gamma_0) \land 1  & x\in\ol{\Omega}_0,\\
(-{\rm dist}(x,\Gamma_0)) \vee (-1) &x\in\R^N\setminus\Omega_0.
\end{array}
\right.
\end{equation*}
It is not difficult to see that~\eqref{u0representant}, 
(I1) and (I2) hold with $\nu(x)=-x.$
Conversely, suppose that (I1) and (I2) with $\nu(x)=-x$ 
hold for some $u_0.$
We claim that $\Omega_0=\{u_0>0\}$ is star-shaped
with respect to the ball $B(0,r_0)$ with
$$
r_0<\frac{\eta_0}{\|Du_0\|_\infty}.
$$
Let $y\in \ol{B}(0,r_0)$ and $x\in \partial\Omega_0$
and define $g(\lam)=u_0((1-\lam)x+\lam y).$
It suffices to show that $g>0$ on $(0,1].$
From (I1), (I2), we have
\begin{eqnarray*}
g(\lam)=u_0((1-\lam)x+\lam y)\geq u_0((1-\lam)x)
-\lam\|Du_0\|_\infty r_0 
\geq \eta_0\lam- \lam\|Du_0\|_\infty r_0 >0
\end{eqnarray*}
for $\lam\in (0,\lam_0].$
Let 
$$
\lam^*= {\rm max}\{\lam \in [\lam_0,1] : g>0
\ {\rm on} \ [0,\lam]\}.
$$
If $\lam^*=1,$ then the proof is complete. Otherwise, let
$\ep>0$ small enough such that $\ep/(1-\lam^*+\ep)<\lam_0.$
From (I1), (I2), we have
\begin{eqnarray*}
0=g(\lam^*)&=& u_0((1-\lam^*)x+\lam^*y)\\
&=& u_0 ( (1-\frac{\ep}{1-\lam^*+\ep})((1-\lam^*+\ep)x+(\lam^*-\ep)y)
+ \frac{\ep}{1-\lam^*+\ep}y)\\
&\geq& g(\lam^*-\ep)+
(\eta_0-\|Du_0\|_\infty r_0)\frac{\ep}{1-\lam^*+\ep} >0,
\end{eqnarray*}
which is a contradiction. It completes the proof of the claim.
The proof of the fact that a star-shaped with respect to a ball
domain has a locally Lipschitz continuous boundary may be found
in \cite[Prop. 2.4.4 and Theorem 2.4.7]{hp05} or \cite[Lemma p.20]{mp97}.

We turn to the proof of (ii).
We need to recall some notations and definitions and we refer
the reader to \cite{clsw98} for further details.
The Clarke generalized directional derivative at $x$ in the direction
$h$ of $f:\R^N\to \R$
is
$$
f^\circ (x,h)= \mathop{\rm lim\,sup}_{y\to x, \lambda\downarrow 0}
\frac{f(y+\lambda h)-f(y)}{\lambda}.
$$
The Clarke generalized derivative at $x$ is the closed convex set
$$
\partial f(x)= \{p\in\R^N \; : \;
{\rm for \ all \ } h\in\R^N, \,
\mathop{\rm lim\,sup}_{y\to x, \lambda\downarrow 0}
\frac{f(y+\lambda h)-f(y)-\langle p,h\rangle}{\lambda}\geq 0  \},
$$
which is nonempty when $f$ is locally Lipschitz continuous at $x.$
The Clarke tangent cone to $\Omega$ at $x\in\Gamma$ is the convex cone
$$
T_\Omega (x)=\{h\in\R^N : d_\Omega^\circ (x,h)=0\} 
$$
and the Clarke normal cone is the polar of the latter, i.e.,
$$
N_\Omega (x)= \{ \xi\in\R^N : \langle \xi ,h\rangle \leq 0 \
{\rm for \ all \ } h\in T_\Omega (x)\}.
$$
The proof is divided in several steps.

\noindent
{\it 1. We claim that, for any $x\in\Gamma$ and $\xi\in N_\Omega (x)\cap
\bS^{N-1},$
there exists $\nu_x\in \bS^{N-1}$ such that} 
\begin{eqnarray}\label{ineg-sur-cone}
1\geq \langle \xi,\nu_x \rangle \geq  \frac{1}{\sqrt{K^2+1}}.
\end{eqnarray}
We fix $x\in \Gamma,$ $x=(x',x_N)\in \R^{N-1}\times\R$ and consider
neighborhoods $V'=B(x',r)$ of $x'$ and $V_N=B(x_N,r)$ of $x_N$
such that there exists a $K$-Lipschitz
continuous function $f:V'\to \R$ 
with ${\rm graph}(f)=\Gamma\cap (V'\times V_N)$
and ${\rm epi}\, f \cap (V'\times V_N)= \Omega\cap (V'\times V_N).$
Let 
$\xi =(\xi',\xi_N) \in N_\Omega (x)\cap \bS^{N-1}
= N_{{\rm epi}\, f} (x',f(x'))\cap \bS^{N-1}.$
By definition, for all $v=(v',v_N)\in T_{{\rm epi}\, f} (x',f(x')),$
\begin{eqnarray}\label{bla111}
\langle (v',v_N),(\xi',\xi_N)\rangle
=  \langle v',v_N\rangle + v_N \xi_N \leq 0.
\end{eqnarray}
By \cite[Theorem 2.5.7]{clsw98}, since $f$ is Lipschitz continuous,
$T_{{\rm epi}\, f} (x',f(x'))=
{\rm epi}\, f^\circ (x;\cdot).$ It follows that, for
any $r\geq v_N,$ $(v',r)$ still belongs to $T_{{\rm epi}\, f} (x',f(x')).$
From \eqref{bla111}, we get
\begin{eqnarray}\label{bla2}
 \langle v',v_N\rangle + r \xi_N \leq 0 \quad 
{\rm for \ all \ } r\geq  v_N.
\end{eqnarray}
A first consequence is that $\xi_N\leq 0$ necessarily. Moreover
$$
\xi_N< 0.
$$ 
Indeed, if $\xi_N=0$ then \eqref{bla2}
holds for all $r\in\R$ and therefore 
$(v',r)\in T_{{\rm epi}\, f} (x',f(x'))=
{\rm epi}\, f^\circ (x;\cdot)$ for all $r\in\R.$
Therefore, using that $f$ is $K$-Lipschitz continuous, we get
$$
r\geq f^\circ (x;v')\geq -K|v'|
$$
which leads to a contradiction for $r\to -\infty.$
Using again \cite[Theorem 2.5.7]{clsw98}, we have
$$
(\xi',\xi_N)= -\xi_N (-\frac{\xi'}{\xi_N},-1)\in N_{{\rm epi}\, f} (x',f(x'))
\quad \Rightarrow \quad -\frac{\xi'}{\xi_N}\in \partial f(x).
$$
Since $f$ is $K$-Lipschitz continuous, it follows
$$
\frac{|\xi'|}{|\xi_N|}\leq K
$$
Using that $\xi\in \bS^{N-1},$ we obtain
$$
|\xi_N|\geq \frac{1}{\sqrt{1+K^2}}.
$$
Taking $\nu_x=(0,-1)$ we obtain easily \eqref{ineg-sur-cone}.
\smallskip

We denote by $d_\Omega$ the distance to $\Omega$ and by
$d_\Gamma^s$ the signed distance to $\Gamma$ which is negative
in $\Omega.$
For any set $A\subset\R^N,$ $\overline{\rm co}\, A$ is the
closure of the convex hull of $A.$
\smallskip

\noindent
{\it 2. We claim  $\partial d_\Gamma^s(x)
\subset \overline{\rm co}[N_\Omega (x)\cap \bS^{N-1}]$ for all $x\in\Gamma$.}
(Notice it means that the generalized derivative of the signed distance
does not contain 0.)
Let $x\in \Gamma$ and $x_i$ a sequence of points which converges to
$\Gamma$ such that $d_\Gamma^s$ is differentiable at $x_i.$
Assume that $p_i:= \nabla d_\Gamma^s(x_i)$ converges to $\overline{p}.$
Suppose first that, up to extract a subsequence, 
$x_i\not\in \overline{\Omega}.$
Then, since $d_\Gamma^s=d_\Omega$ in $\R^N\setminus \overline{\Omega},$
and $d_\Omega$ is differentiable at $x_i,$ it means that
$x_i$ has a unique closest point $\overline{x}_i\in \Gamma$ and
$$
 p_i=\nabla d_\Omega(x_i)=\frac{x_i-\overline{x}_i}{|x_i-\overline{x}_i|},
\quad |p_i|=1.
$$
But $N_\Omega (x)$ is the convex hull of the
cone generated by such limits (\cite[Exercise 8.5, p.96]{clsw98}).
Thus $\overline{p}\in N_\Omega (x)\cap \bS^{N-1}.$  
From  \cite[Theorem 2.8.1]{clsw98},
$\partial d_\Gamma^s(x)$ is the convex hull of such $\overline{p}$
which completes the proof of the claim.
\smallskip

\noindent
{\it 3. There exists an open bounded neighborhood $\mathcal{O}$
of $\Gamma$ and $\bar\eta >0$ such that,
for all $y \in \mathcal{O},$
$$
\mathcal{V}(y)=\{ \nu\in\overline{B(0,1)} : 
1\geq \langle p, \nu \rangle \geq \bar\eta 
\ {for \ all } \ p\in \partial d_\Gamma^s(y) \}
$$
is a nonempty compact convex subset of $\R^N.$}
The subsets $\mathcal{V}(y)$ are clearly convex, closed and bounded
for any $1> \bar\eta >0$ and open subset $\mathcal{O}.$ 
It remains to prove that there
are nonempty for some  $\bar\eta.$ It is true
for $\mathcal{V}(x)$ with $\bar\eta=\sqrt{1+K^2}^{-1}$ by Claims 1 and 2.
To extend this property in a neighborhood $B(x,r)$ of $x,$ we notice the
following facts: since $\partial d_\Gamma^s(\cdot)$ is upper-semicontinuous
(\cite[Proposition 2.1.5]{clsw98}) and $0\not\in \partial d_\Gamma^s(x),$
there exists $r>0$ such that $0\not\in \partial d_\Gamma^s(y)$ for all
$y\in B(x,r).$ By Clarke's implicit function theorem  
\cite[Proposition 3.3.6]{clsw98}, the $d_\Gamma^s(y)$-level sets of
$d_\Gamma^s$ are Lipschitz continuous in $B(x,r).$ The Lipschitz
constant is controlled by the distance from 0 to $\partial d_\Gamma^s(y).$
Up to take $r$ small, it depends only on $K.$
Then, we can repeat the previous arguments and obtain the result
in $B(x,r)$
(up to take $0<\bar\eta$ smaller than $1/\sqrt{1+K^2}$).
We then find $\mathcal{O}$ by compactness of $\Gamma.$
\smallskip

\noindent
{\it 4. The multi-valued map $\mathcal{V}:  \mathcal{O}\rightrightarrows\R^N$
is lower-semicontinuous.} Let $y\in \mathcal{O}.$
Let $\nu\in \mathcal{V}(y)$ and
$\epsilon>0.$ Since $\partial d_\Gamma^s(\cdot)$ is upper-semicontinuous,
by definition, there exists $\delta>0$ such that, if $|y-y'|<\delta$ then
$\partial d_\Gamma^s(y')\subset \partial d_\Gamma^s(y)+\epsilon B$
(where $B=\overline{B(0,1)}$).
Any $p'\in \partial d_\Gamma^s(y')$ can be written $p'=p+\epsilon w$
where $p\in \partial d_\Gamma^s(y)$ and $|w|\leq 1.$ 
Using that $\nu\in \mathcal{V}(y),$
it follows
$$
1+\epsilon \geq \langle \nu, p'\rangle =\langle \nu, p\rangle
+\langle \nu, w\rangle\geq \bar\eta -\epsilon
$$
This proves that $\mathcal{V}(y)\subset \mathcal{V}(y')+\epsilon B$
which is the definition of the lower-semicontinuity of a multi-valued
function.

\noindent
{\it 5. Michael's continuous selection theorem and end of the proof.}
From steps 3 and 4, we can apply Michael's continuous selection theorem
(\cite{af90}): there exists a continuous map 
$\Phi : \mathcal{O}\to \R^N$
such that $\Phi(y)\in \mathcal{V}(y),$ i.e.,
for all $y\in \mathcal{O}$ and $p\in \partial d_\Gamma^s(y),$ we have
\begin{equation}\label{ineg888}
d_\Gamma^s(y+\lam \Phi(y))\geq d_\Gamma^s(y)+\langle p, \lam \Phi(y)\rangle
+o_y(\lam)\geq d_\Gamma^s(y)+\eta\lam+o_y(\lam).
\end{equation}
According to Remark~\ref{rmk-C1-lisse-geom},
up to decrease $\eta,$ we may choose $\Phi$ which is smooth, bounded and
Lipschitz continuous. Using the Lipschitz continuity of $d_\Gamma^s$ and $\Phi,$
we may earn some uniformity in \eqref{ineg888} up to reduce $\eta.$
More precisely, for every $x\in\mathcal{O},$ there exists $r_x$ and $\lam_x$
such that
\begin{equation*}
d_\Gamma^s(y+\lam \Phi(y))\geq d_\Gamma^s(y)
+\frac{\bar\eta}{2}\lam,
\end{equation*}
for all $y\in B(x,r_x)\subset \mathcal{O}$ and $\lam\in [0,\lam_x].$
We conclude by compactness of $\Gamma$ (Note that we can modify
the signed distance function far from $\Gamma$ in order to have
a bounded function).
\end{proof}

\begin{figure}[ht]  
\begin{center}  
\epsfig{file=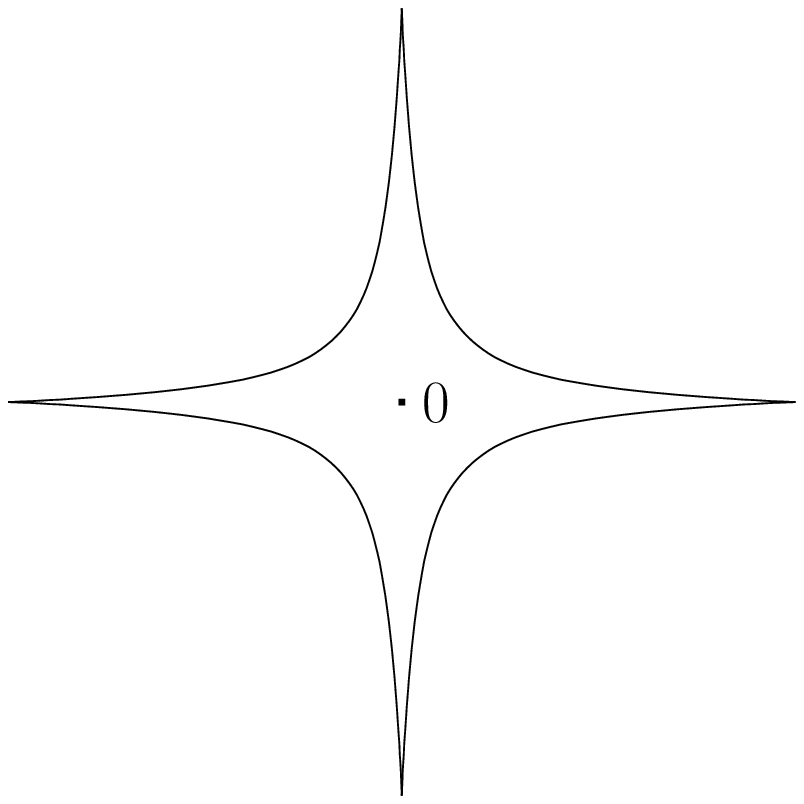, width=5cm}
\hspace*{2cm}  
\epsfig{file=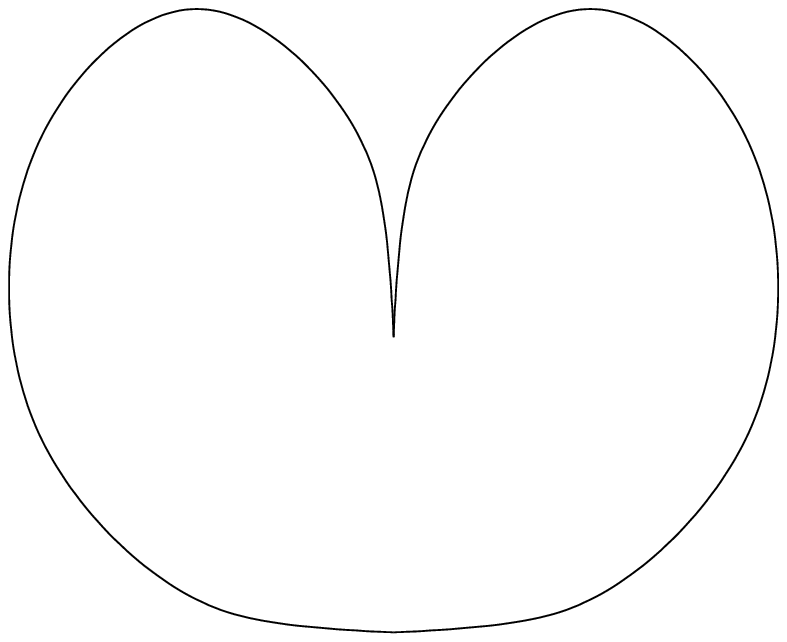, width=4cm}  
\end{center}  
\caption{ \label{dessins-ens}  
{\it A set which is star-shaped with respect to 0 but does not
satisfy (I2) and a set whose boundary is not locally Lipschitz
continuous and which satisfies (I2).
}}  
\end{figure}


\begin{rem}
Star-shaped property
is not sufficient to ensure (I2), see the counterexample of Figure~\ref{dessins-ens}.  There exist some sets which satisfy (I2) but they
do not have a locally Lipschitz boundary, see Figure~\ref{dessins-ens}.
\end{rem}


\begin{proof}[Proof of Proposition {\rm \ref{regularity}}]
Existence of a solution to \eqref{e-2}
is given by Assumption (A8). 
We prove the uniqueness and the Lipschitz continuity regularity in $x$
of the solutions. Let $u\in C(\cQ)$ be 
a solution of \eqref{e-2} which satisfy \eqref{condmoins1}.
Let $K,\ep,\eta>0$ and set 
\[
M:=\sup_{x,y\in\R^{N},t\in[0,T]}
\{u(x,t)-u(y,t)-e^{Kt}\frac{|x-y|^{4}}{\ep^{4}}-\eta t\}. 
\]
Let $M$ be attained at $(\hat{x},\hat{y},\hat{t})\in\R^{2N}\times[0,T]$. 
By (A8), we may assume that 
$(\ol{x},\ol{y},\ol{t})\in \ol{B}(0,R)^{2}\times[0,T]$. 

We first consider the case where $\ol{t}\in(0,T]$. 
In view of Ishii's lemma, for any $\rho>0$, 
there exist 
$(a,p,X)\in \ol{J}^{2,+}u(\ol{x},\ol{t})$ and 
$(b,p,Y)\in \ol{J}^{2,-}u(\ol{y},\ol{t})$ 
such that 
\begin{gather}
a-b=\frac{Ke^{K\ol{t}}|\ol{x}-\ol{y}|^{4}}{\ep^{4}}+\eta, 
\quad 
p=\frac{4e^{K\ol{t}}|\ol{x}-\ol{y}|^2}{\ep^{4}}(\ol{x}-\ol{y}), \nonumber\\ 
\left(
\begin{array}{cc}
X & 0 \\
0 & -Y 
\end{array}
\right)
\le 
A+\rho A^{2}, \label{ishii-lem1}
\end{gather}
where 
\[
A:=
\frac{4e^{K\ol{t}}}{\ep^{4}}|\ol{x}-\ol{y}|^2 
\left(
\begin{array}{cc}
I & -I \\
-I & I 
\end{array}
\right)
+\frac{8e^{K\ol{t}}}{\ep^{4}}
\left(
\begin{array}{cc}
(\ol{x}-\ol{y})\otimes(\ol{x}-\ol{y}) & 
-(\ol{x}-\ol{y})\otimes(\ol{x}-\ol{y}) \\
-(\ol{x}-\ol{y})\otimes(\ol{x}-\ol{y}) & 
(\ol{x}-\ol{y})\otimes(\ol{x}-\ol{y})
\end{array}
\right). 
\]
The definition of viscosity solutions immediately implies the
following inequalities: 
\[
a+H_{\ast}(\ol{x},\ol{t},p,X)\le 0, 
\quad
b+H^{\ast}(\ol{y},\ol{t},p,Y)\ge 0. 
\]
We have 
\begin{equation}\label{ineq-11111}
\frac{Ke^{K\ol{t}}|\ol{x}-\ol{y}|^{4}}{\ep^{4}}+\eta
+H_{\ast}(\ol{x},\ol{t},p,X)-H^{\ast}(\ol{y},\ol{t},p,Y)\le0. 
\end{equation}

We shall distinguish two cases: 
(i) 
for any $\ep\in(0,1)$, 
$p\not=0$; 
(ii) 
there exist 
$\{\ep_j\}_{j\in\N}$ 
such that 
$\ep_j\to0$ as $j\to\infty$, 
$p=0$ 
for any $j\in\N$.

We first consider case (i). 
By (A7) with $\lam=0$, we have 
\begin{align*}
\frac{Ke^{K\ol{t}}|\ol{x}-\ol{y}|^{4}}{\ep^{4}}+\eta
\le \ & 
H(\ol{y},\ol{t},p,Y)-H(\ol{x},\ol{t},p,X)\\
{} = \ & 
H(\ol{y},\ol{t},\frac{4e^{K\ol{t}}|x-y|^{2}}{\ep^{4}}(x-y),Y)
-H(\ol{x},\ol{t},\frac{4e^{K\ol{t}}|x-y|^{2}}{\ep^{4}}(x-y),X)\\
{} \le \ & 
C_{H}\Bigl(\frac{e^{K\ol{t}}|\ol{x}-\ol{y}|^{4}}{\ep^{4}}
+\rho\|A\|^2\Bigr). 
\end{align*}

In case (ii), 
we have $\ol{x}=\ol{y}$. 
Therefore we have 
$A=0$, $X\le0$, $Y\ge0$ and 
$\eta 
\le 
H_{\ast}(\ol{y},\ol{t},0,Y)-H^{\ast}(\ol{x},\ol{t},0,X)
\le 
H_{\ast}(\ol{y},\ol{t},0,0)-H^{\ast}(\ol{x},\ol{t},0,0)
=0$, 
which is a contradiction.

Therefore, 
sending $\rho\to0$ and 
setting 
$K=C_{H}$, 
necessarily we have 
$\ol{t}=0$. 
We get for all $x,y\in\R^N$, $t\in[0,T]$ 
\begin{align*}
{}&
u(x,t)-u(y,t)-\eta t\\
\le \ &
u_{0}(\ol{x})-u_{0}(\ol{y})-\frac{|\ol{x}-\ol{y}|^{4}}{\ep^{4}}
+\frac{e^{Kt}|x-y|^{4}}{\ep^{4}}\\
\le \ &
\|Du_{0}\|_{\infty}|\ol{x}-\ol{y}|
-\frac{|\ol{x}-\ol{y}|^{4}}{\ep^{4}}
+\frac{e^{Kt}|x-y|^{4}}{\ep^{4}}\\
\le \ &
\frac{3}{4^{4/3}}
\|Du_{0}\|_{\infty}^{4/3}\ep^{4/3}
+\frac{e^{Kt}|x-y|^{4}}{\ep^{4}}. 
\end{align*}
We have used the Young inequality in the third inequality. 
Setting 
\[
\ep^{4}
=\Bigl(\frac{ 4^{4/3}
e^{Kt}|x-y|^{4}}{C\|Du_{0}\|_{\infty}^{3/4}}\Bigl)^{3/4}, 
\]
we get 
\[
u(x,t)-u(y,t)-\eta t\le \|Du_{0}\|_{\infty}e^{Kt/4}|x-y|. 
\]
Sending $\eta\to0$, we have 
\[
u(x,t)-u(y,t) \le \|Du_{0}\|_{\infty}e^{Kt/4}|x-y|. 
\]
On the one hand,
by taking $x=y$ in the above inequality, we get $u\leq v$ and obtain
the uniqueness of the solution. On the other hand, by choosing $u=v,$ 
we obtain~\eqref{reg-lip-u-x}.

We now prove \eqref{reg-hold-u-t}.
Set $ R:=2R_{T}\vee\|Du_{0}\|_{\Li(\R)}e^{KT}$, 
where $K$ is the constant given by 
Proposition  \ref{regularity}.
Recalling \eqref{ancien(A6)},
in view of Lemma 9.1 in~\cite{BBL}, there exists 
a constant $\tilde{L}>0$ 
such that 
\[
|u(x,t)-u(x,s)|\le\tilde{L}|t-s|^{1/2}
\]
for all $x\in B(0,R_{T})$, 
$s\in[t,T]$. 
Noting that $u(x,t)\equiv-1$ on $(\R^N\setminus B(0,R_{T}))
\times[0,T]$, 
we get 
\[
|u(x,t)-u(x,s)|\le\tilde{L}|t-s|^{1/2}
\]
for all $x\in \R^{N}$, $t,s\in[0,T]$. 
\end{proof}

\begin{lem}\label{psi}
Assume {\rm (I2)}. 
Define the nondecreasing function $\Psi\in C(\R)$ by 
\begin{eqnarray*}
\Psi(r):=
\left\{
\begin{array}{lcl}
-1 & \textrm{if} & r\le -\frac{3\del_{0}}{4}, \\
\frac{2(2-\del_{0})}{\del_{0}}(r+\frac{\del_{0}}{2})-\frac{\del_{0}}{2} 
& \textrm{if} & -\frac{3\del_{0}}{4}\le r\le -\frac{\del_{0}}{2}, \\
r & \textrm{if} 
& -\frac{\del_{0}}{2}\le r\le\frac{\del_{0}}{2}, \\
\frac{\del_{0}}{2} & \textrm{if} 
& \frac{\del_{0}}{2}\le r. 
\end{array}
\right. 
\end{eqnarray*}
There exists a constant $\ol{\lam}\in(0,\lam_{0}]$ 
which depends only on 
$\del_{0}, \eta_0,$ $\|Du_{0}\|_{\infty}$ and $\|\nu\|_{\infty}$
and satisfies 
\begin{equation}\label{star-ineq-1}
u_{0}(\psi_{\lam}(x))\ge \Psi(u_{0}(x)+\lam\eta_{0}) 
\quad \textrm{for any} \ x\in\R^{N}, \lam\in[0,\ol{\lam}]. 
\end{equation}
\end{lem}

\begin{proof}
It is easy to see inequality \eqref{star-ineq-1} if 
$|u_{0}(x)|\le\del_{0}$ in view of (I2). 
In the case where 
$u_{0}(x)>\del_{0}$, 
we have 
\[
u_{0}(x+\lam\nu(x))
\ge 
u_{0}(x)-\lam\|Du_{0}\|_{\infty}|\nu(x)|
\ge 
\del_{0}-\lam\|Du_{0}\|_{\infty}
\|\nu\|_{\infty}, 
\]
which implies 
inequality \eqref{star-ineq-1} if
$\lam\in[0,\ol{\lam}]$ with 
$\ol{\lam}\|Du_{0}\|_{\infty}\|\nu\|_{\infty}\leq \del_{0}/2.$ 
Finally, we consider the case where 
$u_{0}(x)<-\del_{0}$. 
By replacing $\ol{\lam}$ by a smaller constant 
if necessary, 
we may assume that 
$\lam\eta_{0}\le\del_{0}/4$ 
for all $\lam\in[0,\ol{\lam}]$. 
Then we have $\Psi(u_{0}(x)+\lam\eta_{0})=-1$, 
which yields a conclusion. 
\end{proof}


\bibliographystyle{amsplain}
\providecommand{\bysame}{\leavevmode\hbox to3em{\hrulefill}\thinspace}
\providecommand{\MR}{\relax\ifhmode\unskip\space\fi MR }
\providecommand{\MRhref}[2]{%
  \href{http://www.ams.org/mathscinet-getitem?mr=#1}{#2}
}
\providecommand{\href}[2]{#2}

\end{document}